\documentclass[11pt]{amsart}

\usepackage{enumerate}

\usepackage[foot]{amsaddr}
\usepackage[english]{babel}

\allowdisplaybreaks

\usepackage[
pdftex,hyperfootnotes]{hyperref}
\hypersetup{
	colorlinks=true,
	linkcolor=NavyBlue, 
	urlcolor=RoyalPurple,
	citecolor=OliveGreen,
	pdftitle={Bidiagonal factorization of Hessenberg or Banded matrices},
}
\usepackage[dvipsnames,svgnames,table,x11names]{xcolor}
\usepackage{pgfplots}
\usepackage{tikz}
\usepackage{tikz-3dplot}
\usetikzlibrary{automata,quotes, chains,matrix,calc,shadows,shapes.callouts,shapes.geometric,shapes.misc,positioning,patterns,decorations.shapes,
	decorations.pathmorphing,decorations.markings,decorations.fractals,decorations.pathreplacing,shadings,fadings,arrows.meta,bending}

\usepackage{nicematrix,drawmatrix}
\usepackage[utf8]{inputenc}

\usepackage{comment}
\usepackage{graphicx}
\usepackage[textwidth=17.5cm,textheight=22.275cm,
height=
24.275cm,
width=18cm]{geometry}

\usepackage{amssymb,latexsym,amsmath,amsthm,bm}
\usepackage{mathrsfs}
\usepackage{mathtools,arydshln,mathdots}
\mathtoolsset{showonlyrefs}
\usepackage{tcolorbox}

\usepackage{Baskervaldx}
\usepackage[]{newtxmath}

\makeatletter
\DeclareRobustCommand{\cev}[1]{%
	{\mathpalette\do@cev{#1}}%
}
\newcommand{\do@cev}[2]{%
	\vbox{\offinterlineskip
		\sbox\z@{$\m@th#1 x$}%
		\ialign{##\cr
			\hidewidth\reflectbox{$\m@th#1\vec{}\mkern2mu$}\hidewidth\cr
			\noalign{\kern-\ht\z@}
			$\m@th#1#2$\cr
		}%
	}%
}
\makeatother

\usepackage[all]{xy}
\hypersetup{
	colorlinks=true,
	linkcolor=blue}
\date{}
\providecommand{\abs}[1]{\lvert#1\rvert}

\theoremstyle{plain}

\newtheorem{Theorem}{Theorem}[section]
\newtheorem{Corollary}[Theorem]{Corollary}
\newtheorem{Lemma}[Theorem]{Lemma}
\newtheorem{Proposition}[Theorem]{Proposition}
\newtheorem{Definition}[Theorem]{Definition}
\newtheorem{Remark}[Theorem]{Remark}
\newtheorem{Example}[Theorem]{Example}

\renewcommand{\d}{\operatorname{d}}

\newcommand{\sgn}{\operatorname{sgn}}

\newcommand{\diag}{\operatorname{diag}}

\newcommand{\C}{\mathbb{C}}

\newcommand{\A}{\mathbb{A}}
\newcommand{\N}{\mathbb{N}}

\begin{document}
	
		\title[Uvarov Perturbations for Mixed-Type Multiple Orthogonality]{ Uvarov Perturbations for Multiple Orthogonal Polynomials\\ of the Mixed Type on the Step-Line}

	\author[M Mañas]{Manuel Mañas$^{1}$}
	
	\author[M Rojas]{Miguel Rojas$^{2}$}
	\address{Departamento de Física Teórica, Universidad Complutense de Madrid, Plaza Ciencias 1, 28040-Madrid, Spain}
	\email{$^{1}$manuel.manas@ucm.es}
	\email{$^{2}$migroj01@ucm.es}


\begin{abstract}
 Uvarov-type perturbations for mixed-type multiple orthogonal polynomials on the step line are investigated within a matrix-analytic framework. The transformations considered involve both rational and additive modifications of a rectangular matrix of measures, implemented through left and right multiplication by regular matrix polynomials together with the addition of finitely many discrete matrix masses. These operations induce structured finite-band modifications of the corresponding moment matrix and rational transformations of the associated Markov–Stieltjes matrix function. Explicit Uvarov-type connection formulas are obtained, relating the perturbed and unperturbed families of mixed-type multiple orthogonal polynomials, and determinant conditions ensuring the existence of the new orthogonality are established. The algebraic consequences of these transformations are analyzed, showing their effect on the structure of the block-banded recurrence operators and on the spectral data encoded in the Markov–Stieltjes matrix functions. As an illustrative example, nontrivial Uvarov-type perturbations of the Jacobi–Piñeiro multiple orthogonal polynomials are presented. The results unify Christoffel, Geronimus, and Uvarov transformations within a single matrix framework, providing a linear-algebraic perspective on rational and additive spectral transformations of structured moment matrices and their corresponding recurrence operators.
\end{abstract}

	\maketitle
\tableofcontents
\section{Introduction}

Multiple orthogonal polynomials constitute a rich and flexible class of special functions with strong connections to matrix analysis, operator theory, and spectral methods. Unlike the classical orthogonal polynomials, which are associated with a single scalar weight, multiple orthogonal polynomials (MOPs) involve several measures and weights simultaneously, giving rise to coupled orthogonality relations and to structured block-matrix representations of their recurrence coefficients. These block structures naturally lead to a matrix formulation of orthogonality and to linear-algebraic tools such as Gauss–Borel factorizations, low-rank structured modifications, and block-banded spectral decompositions. The combination of analytic and algebraic perspectives makes multiple orthogonal polynomials a natural object of study within modern linear algebra and its applications to mathematical physics, probability, and numerical analysis.

The theory of MOPs is deeply connected with Hermite–Padé approximation and constructive function theory. Classic sources include Nikishin and Sorokin’s monograph~\cite{nikishin_sorokin} and Van Assche’s chapter in~\cite[Ch.~23]{Ismail}. Their integrable and algebraic aspects were later analyzed from a Gauss–Borel factorization viewpoint in~\cite{afm} and from a dynamical-system perspective in~\cite{andrei_walter}. Asymptotic and analytic properties of their zeros were studied in~\cite{Aptekarev_Kaliaguine_Lopez}, while further developments concerning random matrix models were presented in~\cite{Bleher_Kuijlaars}. Beyond their analytic importance, MOPs have become a cornerstone of the spectral theory of banded block matrices and the algebraic characterization of recurrence operators.

Mixed-type multiple orthogonal polynomials and their associated Riemann–Hilbert problems have appeared in a wide range of applications. They describe non-intersecting Brownian motions and Brownian bridges~\cite{Evi_Arno}, provide explicit solutions to multicomponent Toda and KP-type hierarchies~\cite{adler,afm}, and play a role in number-theoretical constructions, such as Apéry’s proof of the irrationality of $\zeta(3)$~\cite{Apery} and related results for other zeta values~\cite{Ball_Rivoal}, see also \cite{Zudilin}. Their algebraic structure also governs the spectral analysis of block-banded and Hessenberg matrices, as shown in~\cite{aim,phys-scrip,BTP,Contemporary,laa}. Moreover, they have proved essential in modeling Markov chains and random walks with multi-step interactions beyond the classical tridiagonal case~\cite{CRM,finite,hypergeometric,JP}.

From a linear-algebraic standpoint, orthogonal polynomial transformations can be understood as structured algebraic modifications of moment matrices and of the corresponding block-banded recurrence operators. The pioneering transformation of this type was introduced by Christoffel in 1858~\cite{christoffel}, when he investigated the effect of multiplying a measure $\d\mu(x)$ by a polynomial $p(x)=(x-q_1)\cdots(x-q_N)$ in the context of Gaussian quadrature. The resulting Christoffel formulas describe how the new orthogonal polynomials, associated with $\d\hat{\mu}(x)=p(x)\d\mu(x)$, can be expressed as determinants built from the original ones~\cite{Chi,Sze,Gaut}. This transformation has a direct matrix interpretation as a finite-band modification of the moment matrix, preserving its Hankel-type symmetry and inducing a similarity-type transformation on the associated Jacobi or Hessenberg operator.

The inverse operation, introduced by Geronimus~\cite{Geronimus,Maro}, corresponds to dividing the measure by $(x-a)$ and possibly adding a discrete mass at $x=a$. From a matrix point of view, this produces a rational transformation of the Stieltjes (or Markov) function and an additive correction of low rank in the moment matrix. Both Christoffel and Geronimus transformations are particular cases of the broader class of Darboux transformations~\cite{matveev,darboux2,moutard,Bue1,Yoon}, which induce $LU$ or $UL$ factorizations of the recurrence operator. In this sense, orthogonal polynomial transformations serve as matrix factorizations linking two block-banded operators that share most of their spectral data.

A further level of generality was reached with the work of V.~B.~Uvarov~\cite{Uva}, who in 1969 investigated rational modifications of measures of the form
\[
\d\hat{\mu}(x)=\frac{p(x)}{q(x)}\d\mu(x),
\]
where $p(x)$ and $q(x)$ are polynomials whose zeros lie outside the support of $\d\mu$.
This transformation, now known as the \emph{Uvarov transformation}, extends the Christoffel and Geronimus cases to the rational setting and can be interpreted algebraically as a structured rational modification of the moment matrix. In Uvarov’s original formulation no additional discrete masses were included at the zeros of $q(x)$; such terms were later introduced by Maroni and Zhedanov to ensure the existence of a regular orthogonality measure when $q$ vanishes on the support. These generalizations are nowadays referred to as \emph{Uvarov-type} or \emph{additive Uvarov transformations}. They encompass both the classical rational modification considered by Uvarov and its completion through discrete or distributional terms. For an introductory text on these issues see \cite{manas}.

From a matrix-analytic viewpoint, these transformations correspond to rational spectral modifications of the associated Markov–Stieltjes matrix function and to localized algebraic alterations of the moment matrix. They thus preserve most of the block structure of the original banded recurrence operator while altering its spectral data in a controlled algebraic way. This rational framework unifies the multiplicative (Christoffel), divisive (Geronimus), and rational (Uvarov) transformations as specific instances of structured spectral deformations of Hankel-type and block-banded matrices.

In a series of papers~\cite{AAGMM,AGMM,AGMM2,bfm}, a matrix framework for these transformations has been developed. In~\cite{AAGMM} Christoffel perturbations of monic matrix orthogonal polynomials were analyzed, while~\cite{AGMM} addressed general Geronimus transformations and their spectral interpretation via regular matrix polynomials. Later,~\cite{AGMM2} introduced the Geronimus–Uvarov setting and its applications to non-Abelian Toda lattices. The work~\cite{bfm} further extended these ideas to two-weight systems, establishing explicit connection formulas between type~I and type~II MOPs and their associated vector Stieltjes functions.

Recently, other authors have also explored rational and multiplicative perturbations in the context of multiple orthogonality. In particular, Kozhan and Vaktnäs have studied Christoffel transformations for multiple orthogonal polynomials~\cite{KozhanChris} and provided determinantal formulas for rational perturbations of multiple orthogonality measures~\cite{KozhanUvarov}. Their approach focuses mainly on scalar-type rational modifications and on systems of multiple orthogonal polynomials not restricted to the step line, while the perturbation polynomials involved are diagonal and explicit formulas are obtained only for one of the families. By contrast, the present analysis is based on a fully matrix-analytic framework that allows for general left and right matrix polynomial multiplications of a rectangular matrix of measures, leading to non-diagonal and structurally coupled perturbations.

These developments complement the recent series of works~\cite{Manas_Rojas_Christoffel,Manas_Rojas_Geronimus}, where general Christoffel and Geronimus perturbations for mixed-type multiple orthogonal polynomials were introduced. In those papers both left and right matrix polynomial actions on the rectangular matrix of measures were analyzed, explicit connection formulas were established, and the existence of perturbed orthogonality was characterized through the non-cancellation of certain $\tau$–determinants. The present contribution extends that framework to encompass general Uvarov-type perturbations, thereby completing the natural hierarchy of Christoffel, Geronimus, and Uvarov transformations within a unified matrix perspective.

The present paper continues this research line, focusing on \emph{Uvarov-type perturbations for mixed-type multiple orthogonal polynomials} on the step line. From the viewpoint of matrix analysis, additive and rational modifications of a rectangular matrix of measures represented through two regular matrix polynomials $L(x)$ and $R(x)$ are studied, satisfying
\[
\d\hat{\mu}(x) = L(x)\d\mu(x)^{-1}(x) + \sum_{j=1}^{m} W_j\delta(x - x_j),
\]
where each $W_j$ encodes the discrete matrix contribution at $x_j$. This construction leads to localized algebraic modifications of the moment matrix and to rational transformations of the associated matrix Markov–Stieltjes function. Explicit Uvarov-type formulas linking the perturbed and unperturbed multiple orthogonal polynomials are derived, determinant criteria for the existence of perturbed biorthogonality are identified, and the algebraic effect of these structured modifications on the corresponding block-banded recurrence operators is analyzed.

Finally, as an illustrative example, nontrivial Uvarov-type perturbations of the Jacobi–Piñeiro multiple orthogonal polynomials with three weights are examined. This example highlights the interaction between discrete and continuous parts of the modified matrix of measures, the impact on the structure and rank of the moment matrix, and the induced deformation of the spectral data of the underlying operator. The framework developed here unifies multiplicative, divisive, and rational matrix spectral transformations within a common algebraic and analytic setting, strengthening the link between the theory of multiple orthogonality and modern linear algebra.

The paper is organized as follows.
Section~\ref{S:MOPs} recalls the main concepts and notation concerning mixed-type multiple orthogonal polynomials and their representation through moment matrices and Gauss–Borel factorizations.
Section~\ref{S:CSoJC} reviews basic notions on the spectrum and Jordan chains of matrix polynomial perturbations, which play a crucial role in the subsequent analysis.
Next, Section~\ref{S:Matrix_structure_Uvarov} introduces the matrix framework for rational and additive perturbations of measures.
There, the regular matrix polynomials $L(x)$ and $R(x)$ implementing left and right multiplications of the rectangular matrix of measures are defined, and the fundamental relations between the corresponding moment matrices are established.

Section~\ref{S:Uvarov simple} deals with Uvarov-type perturbations associated with simple poles of the rational modification.
Explicit connection formulas linking the perturbed and unperturbed families of mixed-type multiple orthogonal polynomials are derived, and the structure of the induced low-rank corrections in both the moment matrix and the block-banded recurrence operator is identified.
The general situation, including higher-order poles and multiple nodes, is analyzed in Section~\ref{S:Uvarov multiple zeros}, where determinant representations for the perturbed polynomials are obtained, and the existence of the new orthogonality is characterized through the non-vanishing of suitable $\tau$–determinants.
Section~\ref{S:Markov-Stieltjes} is devoted to the study of the associated Markov–Stieltjes matrix function.
It is shown that Uvarov-type perturbations correspond to rational spectral transformations of this function, interpreted as structured algebraic modifications of the underlying block-banded matrices.
A case study is presented in Section~\ref{S:case study}, where the general theory is illustrated with an explicit example involving the Jacobi–Piñeiro multiple orthogonal polynomials.
This example highlights the algebraic and spectral effects of rational and additive perturbations and illustrates how the structure of the recurrence matrices is modified.

Finally, Section~\ref{S:existence} addresses the relation between the non-vanishing of the $\tau$–determinants and the existence of the Uvarov-type transformation.
Unlike the Christoffel case—where sufficiency follows from divisibility theorems for matrix polynomials—and the Geronimus case—where alternative spectral arguments apply—the Uvarov setting remains incomplete.
A partial result is established: sufficiency holds under an additional structural assumption on the banded connection matrix encoding the Uvarov transformation, namely, the existence of an $LU$ factorization (or, equivalently, the invertibility of suitable principal minors) consistent with the non-cancellation of the relevant $\tau$–determinants.
This additional hypothesis appears to be required to ensure the regularity of the transformed moment matrix, and its removal remains an open problem for future research.

A brief summary of conclusions and further research directions closes the paper.

        \subsection{Mixed-Type Multiple Orthogonal Polynomials on the Step Line}\label{S:MOPs}

For  $p,q\in\N$, consider the setup where a rectangular matrix of measures is defined:
\[
\d	\mu=\begin{bNiceMatrix}
	\d\mu_{1,1}&\Cdots &\d\mu_{1,p}\\
	\Vdots & & \Vdots\\
	\d	\mu_{q,1}&\Cdots &\d\mu_{q,p}
\end{bNiceMatrix},
\]
with each \(\mu_{b,a}\) being a measure supported on the interval \(\Delta_{a,b} \subseteq \mathbb{R}\). The support of the matrix of measures is said to be $\Delta\coloneq \cup_{a,b}\Delta_{a,b}$. In this paper we use the notation $\N\coloneq \{1,2,\dots\}$ and $N_0\coloneq \{0,1,2,\dots\}$.

For \(r \in \mathbb{N}\), we define the matrix of monomials:
\[
X_{[r]}(x) = \begin{bNiceMatrix}
	I_r \\
	xI_r \\
	x^2I_r \\
	\Vdots
\end{bNiceMatrix},
\]
and the moment matrix as:
\[
\mathscr{M}\coloneq \int_{\Delta} X_{[q]}(x) \d\mu(x) X_{[p]}^\top(x).
\]

If all leading principal submatrices \(\mathscr{M}^{[k]}\) are invertible, a $LU$ factorization exists:
%
%
\[
\mathscr{M} = S^{-1} H\bar S^{-\top},
\]
where \(S\), \(\bar S\), and \(H\) are lower unitriangular and  invertible diagonal matrices, respectively. 

Next, we define matrix polynomials associated with this factorization:
\[
\begin{aligned}
	B(x) &= S X_{[q]}(x), & A(x) &= X_{[p]}^\top(x) \bar S^\top H^{-1},
\end{aligned}
\]
where \(B(x)\) is monic.
The polynomial entries of \(B\) and \(A\) are:
\[
\begin{aligned}
	B &= \begin{bNiceMatrix}
	B^{(1)}_0 & \Cdots & B^{(q)}_0 \\[2pt] B^{(1)}_1 & \Cdots & B^{(q)}_1 \\[2pt] B^{(1)}_2 & \Cdots & B^{(q)}_2 \\ \Vdots[shorten-end=1.5pt] & & \Vdots[shorten-end=1.5pt] 
\end{bNiceMatrix}, &
A &= \left[\begin{NiceMatrix}
	A^{(1)}_0 & A^{(1)}_1 & A^{(1)}_2 & \Cdots \\ \Vdots & \Vdots & \Vdots & \\ \\A^{(p)}_0 & A^{(p)}_1 & A^{(p)}_2 & \Cdots
\end{NiceMatrix}\right].
\end{aligned}
\]
We will make use of the following notation,  
\[
\begin{aligned}
    B_n(x) & \coloneq \begin{bNiceMatrix}
        B^{(1)}_n(x) & \Cdots & B^{(q)}_n(x)
    \end{bNiceMatrix}, & A_n(x) & \coloneq \begin{bNiceMatrix}
        A_n^{(1)}(x) \\ \Vdots \\ \\A_n^{(p)}(x)
    \end{bNiceMatrix},
\end{aligned}
\]
when referring to the vector polynomials, and 
\[
\begin{aligned}
    \mathcal{B}_N(x) & \coloneq \begin{bNiceMatrix}
        B^{(1)}_{Nq}(x) & \Cdots & B^{(q)}_{Nq}(x) \\
        \Vdots & & \Vdots \\\\
        B^{(1)}_{Nq+q-1}(x) & \Cdots & B^{(q)}_{Nq+q-1}(x)
    \end{bNiceMatrix}, & \mathcal{A}_N(x) & \coloneq \begin{bNiceMatrix}
        A_{Np}^{(1)}(x) & \Cdots & A_{Np+p-1}^{(1)}(x) \\ 
        \Vdots & & \Vdots \\ \\
        A_{Np}^{(p)}(x) & \Cdots & A_{Np+p-1}^{(p)}(x)
    \end{bNiceMatrix},
\end{aligned}
\]
when referring to square matrices. 
The polynomials $B(x)$ and $A(x)$ satisfy 
\[
	\int_{\Delta} B(x)  \d\mu(x) A(x) = I, 
	\]
that lead to the following biorthogonality conditions 
\[
\int_{\Delta} \sum_{b=1}^q \sum_{a=1}^p B^{(b)}_n(x) \d\mu_{b,a}(x) A^{(a)}_m(x) = \delta_{n,m}.
\]

The $LU$ factorization yields:
\[
\begin{aligned}
	\int_\Delta B(x) \, \d\mu(x) X_{[p]}^\top(x) &= H\bar S^{-\top}, & \int_\Delta X_{[q]}(x) \, \d\mu(x) A(x)& = S^{-1},
\end{aligned}
\]
which gives rise to the following orthogonality relations:
\[
\begin{aligned}
	\int_\Delta x^l \sum_{a=1}^p \d\mu_{b,a}(x) A_n^{(a)}(x) &= 0, & b &\in\{1,\dots,q\},& l &\in\left\{0,\dots, \left\lceil\frac{n-b+2}{q}\right\rceil-1\right\} ,\\
\int_\Delta \sum_{b=1}^q B_n^{(b)}(x) \d\mu_{b,a}(x) x^l &= 0,  & a &\in\{1,\dots,p,\}, &l &\in\left\{0,\dots,\left\lceil\frac{n-a+2}{p}\right\rceil-1
\right\} .
\end{aligned}
\]

Consequently, the existence of the $LU$ factorization is equivalent to the existence of the orthogonality.

For \(r \in \mathbb{N}_0\), the shift matrix is:
\[
\Lambda_{[r]} \coloneq 
\left[\begin{NiceMatrix}
	0_r & I_r & 0_r & \Cdots \\ 0_r & 0_r & I_r & \Ddots \\ 0_r & 0_r & 0_r & \Ddots \\ \Vdots & \Ddots[shorten-end=8pt] & \Ddots[shorten-end=10pt] & \Ddots[shorten-end=13pt]
\end{NiceMatrix}\right],
\]
with \(\Lambda_{[1]} \eqcolon \Lambda\) and \(\Lambda_{[r]} = \Lambda^r\). These matrices satisfy:
\[
\Lambda_{[r]}X_{[r]}(x) = x X_{[r]}(x).
\]

            \subsubsection{Christoffel--Darboux Kernels}

We now turn our attention to key elements essential for the constructions in this paper: the kernel polynomials.
\begin{Definition}
	Let us introduce the Christoffel--Darboux (CD) kernel   $K^{[n]}\in \C^{p\times q}(x,y)$ as follows
	\begin{align} \label{CDkernel}
		K^{[n]}(x,y) & \coloneq \begin{bNiceMatrix}
			A_0^{(1)}(x) & \Cdots & A_{n}^{(1)}(x) \\
			\Vdots & & \Vdots \\
			A_0^{(p)}(x) & \Cdots & A_{n}^{(p)}(x)
		\end{bNiceMatrix} \begin{bNiceMatrix}
			B_0^{(1)}(y) & \Cdots & B_0^{(q)}(y) \\
			\Vdots & & \Vdots \\
			B_{n}^{(1)}(y) & \Cdots & B_{n}^{(q)}(y) 
		\end{bNiceMatrix},\\
		K^{[n]}_{a,b}(x,y) & = \sum_{i=0}^{n}A^{(a)}_i(x)B_i^{(b)}(y).
	\end{align}
\end{Definition}

This kernel matrix polynomial satisfies several interesting properties, similar to the  properties of the kernel polynomial in the scalar case.
\begin{Proposition} \label{ProyectorK}
	For any given monic matrix polynomial of degree $N$, i.e., 
	\[P(x) = \sum_{i=0}^N P_i x^i, \quad P_i \in \mathbb{C}^{p\times p}, \quad P_N = I_p \]
	the following projection property holds: 
	\begin{equation} \label{Eq ProyectorK}
		\begin{aligned}
			\int_{\Delta}K^{[n]}(x,t)\d \mu(t) P(t) &= P(x), &  n &\geq Np + p -1.
		\end{aligned}
	\end{equation}
\end{Proposition}

Christoffel–Darboux formulas of the type described here were discussed in \cite{Evi_Arno,AM,afm} for mixed-type multiple orthogonality. For another type of Christoffel–Darboux formulas for mixed-type multiple orthogonality, see \cite{aim}.

            \subsubsection{Cauchy Transforms}
Another crucial set of objects in this construction are the Cauchy transforms of the matrix polynomials under consideration.
\begin{Definition}
	The Cauchy transforms of the orthogonal polynomials $A(x)$ and $B(x)$ are defined as follows:
	\[    \begin{aligned}
		C(z) & \coloneq \int_\Delta \frac{\d \mu(x)}{z-x}A(x), & D(z) & \coloneq \int_\Delta B(x) \frac{\d \mu(x)}{z-x}.
	\end{aligned}\]
	
\end{Definition}
\begin{Remark}
	Note that $C(z)$ and $D(z)$ take values in $\C^{q\times \infty}$ and $\C^{\infty\times p}$, respectively. 
\end{Remark}

\begin{Lemma}
	The entries of the Cauchy transform matrices are holomorphic in  $\C\setminus \Delta$.
\end{Lemma}
\begin{Proposition}
	$C(z)$ and $D(z)$ can also be expressed in terms of the Gauss--Borel matrices, $S$ and $\Bar{S}$:
	\[ \begin{aligned}
		C(z) & = \frac{1}{z} X_{[q]}^\top(z^{-1}) S^{-1}, & D(z) & = \frac{1}{z} H \Bar{S}^{-\top} X_{[p]}(z^{-1}),
	\end{aligned}\]
	whenever $\abs{z} > \mathrm{sup}\{ \abs{x}: x \in \Delta \}$.
\end{Proposition}
\begin{proof}
	We have
	\begin{align*}
		C(z) & = \int_\Delta \frac{\d \mu(x)}{z-x}A(x) = \frac{1}{z} \int_\Delta \frac{\d \mu(x)}{1-\frac{x}{z}}A(x) = \frac{1}{z}\int_\Delta \sum_{i=0}^\infty \left( \frac{x}{z} \right)^i \d \mu(x) A(x) \\[3pt]
		& = \frac{1}{z} X_{[q]}^\top(z^{-1})\int_\Delta X_{[q]}(x) \d \mu(x) X^\top_{[p]}(x) \Bar{S}^\top H^{-1} = \frac{1}{z} X_{[q]}(z^{-1}) S^{-1}.
	\end{align*}
	The expansion of $(1-\frac{x}{z})^{-1}$ in a power series is valid since $\abs{z} > \mathrm{sup}\{ \abs{x}: x \in \Delta \}$. A similar proof can be given for $D(z)$. 
\end{proof}
Let us now introduce two additional families of CD kernels.
\begin{Definition}
	The mixed-type Christoffel--Darboux kernels are defined as follows:
	\begin{align*}
		K^{[n]}_{C}(x,y) & \coloneq C^{[n]}(x)B^{[n]}(y) = \int_\Delta \frac{\d \mu(t)}{x-t}K^{[n]}(t,y), \\
		K^{[n]}_{D}(x,y) & \coloneq A^{[n]}(x)D^{[n]}(y) = \int_\Delta K^{[n]}(x,t)\frac{\d \mu(t)}{y-t}.
	\end{align*}
\end{Definition}

\begin{Remark}
Note that $	K^{[n]}_{C}(x,y) $ and $	K^{[n]}_{D}(x,y) $ take values in $\C^{q\times q}$ and $\C^{p\times p}$, respectively.
\end{Remark}

        \subsection{Canonical Set of Jordan Chains and Divisibility for Matrix Polynomials} \label{S:CSoJC}

Building on \cite{MatrixPoly}, we explore essential results concerning matrix polynomials, crucial for our further study. We focus on regular matrix polynomials; that is,

\begin{equation}\label{eq:R(x)}
	\begin{aligned}
	R(x)& = \sum_{l=0}^{N^R }R_l x^l, & R_l &\in \mathbb{C}^{p \times p},
\end{aligned}
\end{equation}
where \(\det R(x)\) is not identically zero. We have that \(\deg\det R(x)\in\{0,1,\dots,N_Rp\}\), and if $R_{N^R}$ is no singular it follows that  \(\deg\det R(x)=N^Rp\). However, in this paper we  use matrix polynomials with singular leading coefficiens, which  have determinants with lower degree than $Np$ Let us introduce a nonnegative integer  $r$ to model this, so that
\[
\begin{aligned}
	\deg \det R(x) &= N^Rp - r, & r \in & \{0,\dots,Np-1\}.
\end{aligned}
\]
The eigenvalues of \(R(x)\) are the zeros of \(\det R(x)\).

\begin{Proposition}[Smith Form]\label{SmithForm}
	Any matrix polynomial can be represented in its Smith form as:
	\[
	R(x) = E(x)D(x)F(x),
	\]
	where \( E(x) \) and \( F(x) \) are matrices with constant determinants, and \( D(x) \) is a diagonal matrix polynomial. Explicitly, \( D(x) \) takes the form:
	\[
	D(x) = \diag \left(
	\prod_{i=1}^{M^R}(x-\rho_i)^{\kappa_{i,1}}, \prod_{i=1}^{M^R}(x-\rho_i)^{\kappa_{i,2}}, \dots , \prod_{i=1}^{M^R}(x-x_i)^{\kappa_{i,p}}\right),
	\]
	where he eigenvalues \( \rho_i \) are the roots of \( \det R(x) \), and \( \kappa_{i,j} \) are the partial multiplicities. For a root \( \rho_i \) with multiplicity \( K_i \), the following relations hold:
	\[
	\begin{aligned}
		K_i &= \sum_{j=1}^{p} \kappa_{i,j}, & \sum_{i=1}^{M^R}\sum_{j=1}^p \kappa_{i,j} &= N^Rp - r,
	\end{aligned}
	\]
	with \( M^R\) being the number of distinct roots and some partial multiplicities potentially zero.
\end{Proposition}
To simplify the notation, we consider a single eigenvalue \( \rho_0 \) with multiplicity \( K \).

\begin{Definition}[Jordan Chains]
	\begin{enumerate}
		\item A left Jordan chain for \( R(x) \) at \( \rho_0 \in \mathbb{C} \) consists of \( \ell+1 \) row vectors $\{\cev{\boldsymbol w}_0,\cev{\boldsymbol w}_1,\dots,\cev{\boldsymbol w}_\ell\}\subset(\C^p)^*$ satisfying:
		\[
	\begin{aligned}
			\sum_{k=0}^i \frac{1}{k!}\cev{\boldsymbol{w}}_{i-k}R^{(k)}(\rho_0) &= 0, & i &\in \{0, \dots ,\ell\}.
	\end{aligned}
		\]
			\item A right Jordan chain for \( R(x) \) at \( \rho_0 \in \mathbb{C} \) consists of \( \ell+1 \) column vectors $\{\vec{\boldsymbol v}_0,\vec{\boldsymbol v}_1,\dots,\vec{\boldsymbol v}_\ell\}\subset\C^p$ satisfying:
		\[
		\begin{aligned}
			\sum_{k=0}^i \frac{1}{k!}R^{(k)}(x_0) \vec{\boldsymbol{v}}_{i-k}&= 0, & i &\in \{0, \dots ,\ell\}.
		\end{aligned}
		\]
		\item A canonical set of left Jordan chains of \( R(x) \) at \( \rho_0 \) consists of \( K \) vectors structured as:
		\[
		\{ \cev{\boldsymbol{w}}_{1,0}, \cev{\boldsymbol{w}}_{1,1}, \dots, \cev{\boldsymbol{w}}_{1,\kappa_1-1}, \dots, \cev{\boldsymbol{w}}_{s,0}, \dots, \cev{\boldsymbol{w}}_{s,\kappa_s-1} \},
		\]
		where \( s \leq p \) and \( \sum_{i=1}^s \kappa_i = K \). Each subset \( \{ \cev{\boldsymbol{w}}_{i,0}, \dots, \cev{\boldsymbol{w}}_{i,\kappa_i-1} \} \) forms a left Jordan chain of length \( \kappa_i \), with the row vectors \(\cev {\boldsymbol{w}}_{i,0} \) being linearly independent. Similarly, one has  canonical set of right Jordan chains 
			\[
		\{ \vec{\boldsymbol{v}}_{1,0}, \vec{\boldsymbol{v}}_{1,1}, \dots, \vec{\boldsymbol{v}}_{1,\kappa_1-1}, \dots, \vec{\boldsymbol{v}}_{s,0}, \dots, \vec{\boldsymbol{v}}_{s,\kappa_s-1} \},
		\]
	\end{enumerate}
\end{Definition}
Lastly, let's introduce a theorem concerning the divisibility of matrix polynomials, see
	\cite[Corollary 7.11., pag. 203]{MatrixPoly}.
	\begin{Theorem}[Matrix Polynomials Divisibility]\label{Theorem}
Let us consider two regular $p \times p$ matrix polynomials $R(x)$ and $A(x)$. Then, $R(x)$ is a right/left divisor of $A(x)$ if and only if each Jordan chain of $R(x)$ coincides with a Jordan chain of $A(x)$ having the same eigenvalue.
	\end{Theorem}
	\subsection{The Matrix Structure of the Polynomial Perturbation}\label{S:Matrix_structure_Uvarov}
Let's  consider  matrix polynomials  as in \eqref{eq:R(x)}
with  leading and sub-leading matrix coefficients of the form:
\begin{align} \tag{C1}\label{Cond: LeadingMatrixConditions}
	R_{N^R}  & = \begin{bNiceArray}{cw{c}{1cm}c||w{c}{2cm}c}
		\Block{3-2}{0_{(p-r)\times r} }	&	& \Block{3-3}{\left[ t_{N^R} \right]_{(p-r)\times (p-r)}} &&\\\\\\
		\Block{2-2}{	0_{r\times r} }& &\Block{2-3}{0_{r \times (p-r)}} &&\\\\
	\end{bNiceArray}, & R_{N^R-1}  & = \begin{bmatrix}
		\left[ R^1_{N^R-1} \right]_{(p-r)\times r} & \left[ R^2_{N^R-1} \right]_{(p-r)\times (p-r)} \\ \\
		\left[ t_{N^R-1} \right]_{r\times r} & \left[ R^4_{N^R-1} \right]_{r \times (p-r)}
	\end{bmatrix},
\end{align}
where $r$  take values in $\{0,\dots,p-1\}$, and $\left[ t_{N^R} \right]_{(p-r)\times (p-r)}$ and $\left[ t_{N^R-1} \right]_{r\times r}$ are upper triangular matrices with nonzero determinant.

From this point onward, we will consider matrix polynomials in which the leading and sub-leading matrices, $\left[ t_{N^R} \right]_{(p-r)\times (p-r)}$ and $\left[ t_{N^R-1} \right]_{r\times r}$, are chosen to be the identity matrix:
\begin{align} \tag{C2.1}\label{Cond: CondicionesMatricesLideresFinales_1}
	R_{N^R}  & = \begin{bmatrix}
		0_{(p-r)\times r} & I_{(p-r)\times (p-r)} \\ \\
		0_{r\times r} & 0_{r \times (p-r)}
	\end{bmatrix}, & R_{N^R-1}  & = \begin{bmatrix}
		\left[ R^1_{N^R-1} \right]_{(p-r)\times r} & \left[ R^2_{N^R-1} \right]_{(p-r)\times (p-r)} \\ \\
		I_{r\times r} & \left[ R^4_{N^R-1} \right]_{r \times (p-r)}
	\end{bmatrix}.
\end{align}
It is straightforward to observe that any matrix polynomial of the form in condition \eqref{Cond: LeadingMatrixConditions} can be expressed as the product of $R(x)$ and another upper triangular matrix with a nonzero determinant. While we will focus on perturbations where the leading matrices satisfy the condition \eqref{Cond: CondicionesMatricesLideresFinales_1}, multiplying the weight matrix by a matrix with a nonzero determinant will preserve orthogonality, and the newly perturbed polynomials will be linear combinations of the original ones.

\begin{Proposition} \label{Prop2}
The determinant of a matrix polynomial, where the leading and sub-leading matrices satisfy the conditions in \eqref{Cond: CondicionesMatricesLideresFinales_1}, is a polynomial of degree $N^Rp - r$.
\end{Proposition}
\begin{proof}
Expanding the determinant we get
	\begin{align*}
		\det  R(x) & =  \begin{vmatrix}
			\left[ R^1_{N^R-1} \right]x^{N^R-1}+ \text{l.o.t.} & x^{N^R}I_{(p-r)} + \text{l.o.t.}\\
			\\
			x^{N^R-1}I_{r} +  \text{l.o.t.}& \left[ R^3_{N^R-1} \right]x^{N^R-1}+ \text{l.o.t}
		\end{vmatrix} \\
		& = \sum_{\sigma \in \mathcal{S}_p} \text{sgn}(\sigma) R_{1\sigma(1)}R_{2\sigma(2)} \cdots R_{p\sigma(p)} \\&
			= \text{sgn}(\Tilde{\sigma}) R_{1\Tilde{\sigma}(1)}R_{2\Tilde{\sigma}(2)} \cdots R_{(p-r)\Tilde{\sigma}(p-r)} \cdots R_{p\Tilde{\sigma}(p)} 
			+ \sum_{\sigma \neq \Tilde{\sigma}} \text{sgn}(\sigma) R_{1\sigma(1)}R_{2\sigma(2)} \cdots R_{p\sigma(p)},
	\end{align*}
	where the permutation $\Tilde{\sigma}(i)$ is such that $i \rightarrow r+i$ for $i \leq (p-r)$ and $i \rightarrow i-(p-r)$ for $i \geq (p-r)+1$. Here l.o..t  stands for lower order degree terms. This term in the determinant expansion leads to a contribution of the form:
	\begin{multline*}
		\sgn(\Tilde{\sigma}) R_{1\Tilde{\sigma}(1)}R_{2\Tilde{\sigma}(2)} \cdots R_{(p-r)\Tilde{\sigma}(p-r)} \cdots R_{p\Tilde{\sigma}(p)} \\\begin{aligned}
			&= \sgn(\Tilde{\sigma}) R_{1,(r+1)}R_{2,(r+2)} \cdots R_{(p-r),p} R_{(p-r+1),1} \cdots R_{pr} =
			\sgn(\Tilde{\sigma}) (x^{N^R})^{(p-r)}(x^{N^R-1})^{r} = x^{N^Rp-r}.
		\end{aligned}
	\end{multline*} 	
	Any other permutation either results in zero or produces terms of lower degree.
\end{proof}
\begin{Proposition}
$R\left( \Lambda_{[p]}^\top \right)$ is banded lower triangular matrix that from the $N^Rp-r$ subdiagonal is populated with zeros.
\end{Proposition}
\begin{proof}
We have
\begin{equation*}
	R\left( \Lambda_{[p]}^\top \right) = \left[\begin{NiceMatrix}
		R_0 & 0_p & 0_p & \Cdots \\
		R_1 & R_0 & 0_p & \Cdots \\
		\Vdots & & \Ddots & \\
		R_{N^R} & R_{N^R-1} & \Cdots & R_0 & \Cdots \\
		0_p & R_{^R}& R_{N^R-1} & \Cdots & R_0 & \Cdots\\
		\Vdots & \Ddots[shorten-end=-15pt] &  \Ddots[shorten-end=-10pt] &  \Ddots[shorten-end=-25pt] &  & \Ddots
	\end{NiceMatrix} \right].
\end{equation*}
For the block 
\[\begin{bNiceMatrix}
	R_{N^R-1} & R_{N^R-2} \\
	R_{N^R} & R_{N^R-1}  
\end{bNiceMatrix}
\]
we found the following detailed structure
\[\begin{bNiceArray}[small]{w{c}{1.35cm}w{c}{1cm}w{c}{1.5cm}w{c}{1cm}w{c}{1cm}w{c}{1cm}w{c}{1.35cm}w{c}{1cm}w{c}{1cm}w{c}{1.5cm}}[hvlines,margin,small,cell-space-limits=15pt]
	\Block{3-2}{\left[ R^1_{N^R-1} \right]_{(p-r)\times r} } 	& & \Block{3-3}<\large>{\left[ R^2_{N^R-1} \right]_{(p-r)\times (p-r)}} 	&  &&\Block{5-5}<\Huge>{R_{N^R-2}} & &&& \\
	& & &  &&& &&& \\
	& & & &&& &&& \\\cline{1-5}
\Block{2-2}<\LARGE>{I_{r} } 	& & \Block{2-3}{\left[ R^4_{N^R-1} \right]_{r \times (p-r)} } & &&& &&& \\
	& & &  &&& &&& \\
	\hline
	\Block{3-2}<\Large>{0_{(p-r)\times r} } 	& & \Block{3-3}<\huge>{ I_{(p-r)}} 	&  &&\Block{3-2}{\left[ R^1_{N^R-1} \right]_{(p-r)\times r} } & &\Block{3-3}<\large>{\left[ R^2_{N^R-1} \right]_{(p-r)\times (p-r)}} && \\
& & &  &&& &&& \\
& & & &&& &&& \\\cline{1-5}
\Block{2-2}<\LARGE>{0_{r}} 	& & \Block{2-3}<\Large>{0_{r \times (p-r)} } & &&\Block{2-2}<\LARGE>{I_r} & &\Block{2-3}{  \left[ R^4_{N^R-1} \right]_{r \times (p-r)}} && \\
& & &  &&& &&& 
\end{bNiceArray}\]
which has $p-r$ subdiagonals with at least a nonzero entry. Up to the matrix $R_{N-1}$, there will be $N-1$ matrices, which sum up to a total of $Np-r$ subdiagonals.
\end{proof}\setlength{\extrarowheight}{0mm}

\begin{Corollary} \label{Col: ProycK_2}
    If the leading matrix of a given polynomial of degree $N$ is of the form: 
    \begin{equation*}
        P_N = \begin{bNiceMatrix}
            0_{(p-r) \times r} & I_{(p-r)} \\[2pt]
            0_r & 0_{r \times (p-r)}
        \end{bNiceMatrix},
    \end{equation*}
    then the projection property shown in Proposition \ref{ProyectorK}  holds for $n \geq Np + p -1 - r$.
\begin{proof}
    It is a direct consequence of the proof of the stated Proposition, since $\Pi_n$ does not change the structure of 
    \begin{equation*}
        \begin{bNiceMatrix}
            P_0 \\ P_1 \\ \Vdots \\ P_N \\ 0_p \\ \Vdots 
        \end{bNiceMatrix},
    \end{equation*}
    provided that $n+1 \geq Np + p - r$.
\end{proof}
\end{Corollary}
\begin{Remark}
    Similar properties derived in this section hold if we consider the following matrix polynomial: 
    \begin{align*}
    \begin{aligned}
			L(x) & = \sum_{l=0}^{N^L} L_l x^l , & L_l \in \mathbb{C}^{p \times p},
    \end{aligned}
    \end{align*}
    whose leading matrices are
    \begin{align}  \tag{C2.2} \label{Cond: CondicionesMatricesLideresFinales_2}
	\begin{aligned}
			L_{N^L}  & = \begin{bmatrix}
			0_{r\times (q-l)} & 0_{l} \\ \\
			I_{(q-l)} & 0_{(q-l) \times l}
		\end{bmatrix}, & L_{N^L-1}  & = \begin{bmatrix}
			\left[ L^1_{N^L-1} \right]_{l\times (q-l)} & I_{l} \\ \\
			\left[ L^3_{N^L-1} \right]_{(q-l)} & \left[ L^4_{N^L-1} \right]_{(q-l) \times l}
		\end{bmatrix}.
	\end{aligned}
    \end{align}
    For this case, the determinant of the matrix $L(x)$ is of degree $N^Lq-l$.
\end{Remark}
    \section{Uvarov Perturbation}\label{S:Uvarov}
Given regular matrix polynomials $R(x)\in\C^{p\times p}(x)$  and $L(x)\in\C^{q\times q}(x)$, with $\deg R(x)=N^R$ and $\deg L(x)=N^L$,  that satisfy Conditions \eqref{Cond: CondicionesMatricesLideresFinales_1} and \eqref{Cond: CondicionesMatricesLideresFinales_2}, respectively, an Uvarov perturbation of a $q\times p$ matrix of measures is defined by the relation  
\begin{equation*}
    \d \tilde{\mu}(x) R(x) = L(x) \d \mu(x).
\end{equation*}
Note that \( \deg\det R(x)=N^Rp-r  \) and \(\deg\det L(x) = N^Lq-l \) . 
We denote by $\rho_i$ its eigenvalues and by  \( \cev{\boldsymbol{r}}_{i,j,k} ,\vec{\boldsymbol{r}}_{i,j,k} \)   corresponding left and right generalized eigenvectors with $p$ entries, respectively,  associated with corresponding  canonical set of left and right Jordan chains of \( R(x) \). Similarly,  let $\lambda_i$ be the eigenvalues and   \( \cev{\boldsymbol{l}}_{i,j,k} ,\vec{\boldsymbol{l}}_{i,j,k} \)  be the left and right generalized eigenvectors, with $q$ entries, respectively,  associated with corresponding  canonical set of left and right Jordan chains of \( L(x) \).

 The perturbed linear functional admits the decomposition:
\begin{equation}
\label{Eq:GeneralMeasure}    
\d\tilde{\mu}(x) = L(x) \d\mu(x)R^{-1}(x) + \sum_{i=1}^{M^R}\sum_{j=1}^{s^R_i}\sum_{k=0}^{\kappa^R_{i,j}-1}L(x)\boldsymbol{\xi}_{i,j,k}(x) \left( \sum_{l=0}^k\frac{(-1)^l}{l!} \cev{\boldsymbol{r}}_{i,j;k-l}\delta^{(l)}(x-\rho_i)\right),
\end{equation}
where \( \delta \) denotes the Dirac delta distribution and \( \boldsymbol{\xi}_{i,j,k}(x) \) are arbitrary \( q \)-dimensional column vector functions. 

\begin{Remark}\label{Remark: locally integrable}
	We must ensure that the expression \( L(x)\, \d\mu(x)\, R^{-1}(x) \) is a well-defined measure; that is, it must be locally integrable. For instance, if the original measure is absolutely continuous with respect to Lebesgue measure, \( \d\mu = w(x)\, \d x \), then we require that the entries of the matrix function \( L(x)\, w(x)\, R^{-1}(x) \) be locally integrable at the points in \( \sigma(R) \cap \Delta \). This condition is automatically satisfied if the spectrum of \( R(x) \) is disjoint from the support \( \Delta = \operatorname{supp}(\d\mu) \); that is, if \( \sigma(R) \cap \Delta = \varnothing \). However, this requirement is more restrictive than necessary, and in the example we shall analyze, it is not satisfied.
\end{Remark}

\begin{Remark}
There is a slight abuse of notation here: although $\d\tilde{\mu}$ is technically a linear functional, it can be interpreted as a measure only in the special case where it involves delta functions but no derivatives thereof.
\end{Remark}

\subsection{Simple Eigenvalues}\label{S:Uvarov simple}
We first consider the case where all zeros of \( R(x) \) and \( L(x) \) are simple. In this scenario, \eqref{Eq:GeneralMeasure} simplifies to:
\begin{equation}
\label{Eq:SimpleEigMeasure} 
\d\tilde{\mu}(x) = L(x) \d\mu(x)R^{-1}(x) + \sum_{i=1}^{M^R} L(x)\boldsymbol{\xi}_{i}(x) \cev{\boldsymbol{r}}_i\delta(x-\rho_i). 
\end{equation}
Since no derivatives of the delta distribution appear, the Uvarov perturbation of the measure yields a measure—with support with a possibly discrete part—rather than a more general linear functional. Note that \( M^R = N^Rp-r \) and \( M^L = N^Lq-l \) in this case.

\begin{Proposition} \label{Prop:ConnectM}
    The moment matrices satisfy the relation:
    \begin{equation*}
        \tilde{\mathscr{M}}R(\Lambda_{[p]}^\top) = L(\Lambda_{[q]})\mathscr{M}.
    \end{equation*}
\end{Proposition}
\begin{proof}
    Using the property \( \Lambda_{[s]}X_{[s]}(x) = x X_{[s]}(x) \), we compute:
    \begin{align*}
        \tilde{\mathscr{M}}R(\Lambda_{[p]}^\top) &= \int_\Delta X_{[q]}(x) \d \tilde{\mu}(x) X^\top_{[p]}(x) \sum_{i=0}^{N^R} R_i(\Lambda_{[p]}^\top)^i = \int_\Delta X_{[q]}(x) \d \tilde{\mu}(x) \sum_{i=0}^{N^R} R_i X^\top_{[p]}(x)(\Lambda_{[p]}^\top)^i \\
        &= \int_\Delta X_{[q]}(x) \d \tilde{\mu}(x) \left( \sum_{i=0}^{N^R} R_i x^i \right) X^\top_{[p]}(x) = \int_\Delta X_{[q]}(x) \left( \sum_{i=0}^{N^L} L_i x^i \right) \d \mu(x) X^\top_{[p]}(x) \\ 
        & =  \sum_{i=0}^{N^L} L_i(\Lambda_{[q]})^i \int_\Delta X_{[q]}(x) \d \tilde{\mu}(x) X^\top_{[p]}(x) = L(\Lambda_{[q]})\mathscr{M},
    \end{align*}
     where we have used the bimodule structure of \( \mathbb{C}^{p \times \infty}[x] \) and \( \mathbb{C}^{q \times \infty}[x] \) over \( \mathbb{C}^{p \times p}[x] \) and \( \mathbb{C}^{q \times q}[x] \), respectively.
\end{proof}

Assuming the existence of Gauss--Borel factorizations for both moment matrices:
\[
\begin{aligned}
    \tilde{\mathscr{M}} &= \tilde{S}^{-1}\tilde{H} \tilde{\Bar{S}}^{-\top}, & \mathscr{M} &= S^{-1}H \Bar{S}^{-\top},
\end{aligned}
\]
we can relate the original and perturbed orthogonal polynomials through a connection matrix.

\begin{Proposition} 
  We have the following relation
    \begin{equation*}
        \tilde{S}L(\Lambda_{[q]})S^{-1} = \tilde{H}\tilde{\Bar{S}}^{\top} R(\Lambda_{[p]}^\top) \Bar{S}^{-\top} H^{-1}.
    \end{equation*}
    \end{Proposition}
    \begin{proof}
    	The equality follows from Gauss--Borel factorizations and Proposition \ref{Prop:ConnectM}. 
    	\end{proof}
    	
    	\begin{Definition} \label{Prop: OmegaDefinition}
    		The connection matrix \( \Omega \) is
    		\begin{equation*}
    		\Omega\coloneq	\tilde{S}L(\Lambda_{[q]})S^{-1} = \tilde{H}\tilde{\Bar{S}}^{\top} R(\Lambda_{[p]}^\top) \Bar{S}^{-\top} H^{-1}.
    		\end{equation*}
    	\end{Definition}
    	
    \begin{Proposition}
    The connection matrix \(\Omega\) is a banded matrix with at most \( M^R \) nonzero subdiagonals and \( M^L \) nonzero superdiagonals, with ones on its last superdiagonal. Explicitly:
   \begin{equation*}
	\Omega = \begin{bNiceMatrix}
		\Omega_{0,0} & \Omega_{0,1} & \Cdots & \Omega_{0,M^L-1} & 1 & 0 & \Cdots[shorten-end=5pt] \\
		\Omega_{1,0} & \Omega_{1,1} & \Cdots & \Omega_{1,M^L-1} & \Omega_{1,M^L} & 1 & \Ddots[shorten-end=5pt] \\
		\Vdots & \Vdots & & \Vdots & & \Ddots[shorten-end=-25pt] & \Ddots[shorten-end=5pt] \\
		\Omega_{M^R,0} & \Omega_{M^R,1} & \Cdots & \Omega_{M^R,M^L-1} & \Cdots[shorten-end=5pt] & & \Omega_{M^R,M^R+M^L-1} & 1 &  &  \\
		0 & \Ddots[shorten-end=-50pt] & & & & &  &\!\! \Ddots[shorten-end=-8pt] & \Ddots[shorten-end=5pt,] \\
		&  \Ddots[shorten-end=-15pt] & & & & & & &  & 
	\end{bNiceMatrix}.
\end{equation*}
\end{Proposition}

\begin{proof}
    The banded structure results from the properties of \( R(\Lambda_{[p]}^\top) \) and \( L(\Lambda_{[q]}) \), which have \( M^R \) subdiagonals and \( M^L \) superdiagonals, respectively. The normalization of the last superdiagonal is a consequence of Condition \eqref{Cond: CondicionesMatricesLideresFinales_2}.
\end{proof}

The following Proposition establishes explicit relations between the original and perturbed orthogonal polynomials.
\begin{Proposition}
    The following relations hold: 
    \begin{align}
        \label{Eq: ConnectB} \Omega B(x) & = \tilde{B}(x)L(x), \\ \label{Eq: ConnectA} \tilde{A}(x) \Omega & = R(x) A(x). 
    \end{align}
\end{Proposition}
\begin{proof}
    We will only prove the Equation \eqref{Eq: ConnectA}, as the other equation can be proven similarly. 
    \begin{equation*}
        \tilde{A}(x) \Omega = X_{[p]}^\top(x) \tilde{\Bar{S}}^{\top} \tilde{H}^{-1} \Omega = X_{[p]}^\top(x) R(\Lambda_{[p]}^\top) \Bar{S}^{-\top} H^{-1} = R(x) X_{[p]}^\top(x)\Bar{S}^{-\top} H^{-1} = R(x) A(x). 
    \end{equation*}
\end{proof}
To obtain connection formulas where the perturbation matrix and perturbed polynomials appear on the same side, we utilize the Cauchy transforms, leading to a linear system for the components of \( \Omega \).
\begin{Proposition}
    In terms of the Cauchy transforms of the original and perturbed orthogonal polynomials, the connection formulas are given by:
    \begin{align}
        \label{Eq:ConnectD}\tilde{D}(z)R(z) & = \Omega D(z) + \int_\Delta \tilde{B}(x) \d \tilde{\mu}(x) \frac{R(z)-R(x)}{z-x}, \\
        \label{Eq:ConnectC} L(z)C(z) & = \tilde{C}(z)\Omega + \int_\Delta \frac{L(z)-L(x)}{z-x} \d \mu(x) A(x).
    \end{align}
\end{Proposition}
\begin{proof}
    To establish the first relation, we consider the Cauchy transform of Equation \eqref{Eq: ConnectB}:
    \begin{equation*}
        \Omega D(z) = \int_\Delta \Omega B(x) \frac{\d \mu(x)}{z-x} = \int_\Delta \tilde{B}(x) \frac{\d \tilde{\mu}(x)}{z-x}R(x). 
    \end{equation*}
    Subtracting $\tilde{D}(z)R(z)$ from both sides yields:
    \begin{equation*}
        \Omega D(z) - \tilde{D}(z)R(z) = \int_\Delta \tilde{B}(x) \d \tilde{\mu}(x)\frac{R(x)-R(z)}{z-x}.
    \end{equation*}
    The stated result follows directly from this expression. Equation \eqref{Eq:ConnectC} can be proven through an analogous procedure.
\end{proof}

In practical applications, we will primarily utilize Equations \eqref{Eq: ConnectB} and \eqref{Eq:ConnectD} to solve the linear system for the entries of $\Omega$.

\begin{Lemma}\label{Lema: R-BivaMP}
    The matrix function $\frac{R(x) - R(y)}{x-y}$ is a bivariate matrix polynomial of degree $N^R-1$ in both variables.
\end{Lemma}

\begin{proof}
    Employing the algebraic identity:
    \begin{equation*}
        x^n - y^n = (x-y)\left( \sum_{i=1}^n x^{n-i}y^{i-1} \right),
    \end{equation*}
    we readily obtain:
    \begin{equation}\label{eq:Rxy}
        \frac{R(x) - R(y)}{x-y} = \sum_{i=1}^N R_i \sum_{j=1}^i \left( x^{i-j}y^{j-1} \right),
    \end{equation}
    which clearly demonstrates the polynomial nature and degree structure in both variables $x$ and $y$.
\end{proof}

\begin{Definition} \label{Def:NotationBD}
    We introduce the following complex numbers: 
\[    \begin{aligned}
        \mathbb{D}_n^{(i)} & \coloneq  D_n(\rho_i)\vec{\boldsymbol{r}}_i, & \mathbb{W}_n^{(i)} & \coloneq B_n(\rho_i)\boldsymbol{\xi}_{i}(\rho_i)\cev{\boldsymbol{r}}_i R'(\rho_i) \vec{\boldsymbol{r}}_i, & \mathbb{B}^{(j)} \coloneqq B_n(\lambda_j)\vec{\boldsymbol{l}}_j, 
    \end{aligned}  \]
    for $n \in \mathbb{N}_0$, $i \in \{1,\dots,M^R \}$ and $j \in \{ 1,\dots,M^L \}$.
\end{Definition}
\begin{Proposition} \label{Prop: OmegaComp}
    For $n \geq M^R$, the entries of the connection matrix $\Omega$  satisfy the following linear system:
    \begin{multline*}
        \begin{bNiceMatrix}
            \Omega_{n,n-M^R} & \Cdots & \Omega_{n,n+M^L-1}
        \end{bNiceMatrix} 
        \begin{bNiceMatrix}
            \mathbb{B}_{n-M^R}^{(1)} & \Cdots & \mathbb{B}_{n-M^R}^{(M^L)} & \mathbb{D}_{n-M^R}^{(1)}-\mathbb{W}_{n-M^R}^{(1)} & \Cdots & \mathbb{D}_{n-M^R}^{(M^R)}-\mathbb{W}_{n-M^R}^{(M^R)} \\ 
            \Vdots & & \Vdots & \Vdots & & \Vdots \\ \\\\
            \mathbb{B}_{n+M^L-1}^{(1)} & \Cdots & \mathbb{B}_{n+M^L-1}^{(M^L)} & \mathbb{D}_{n+M^L-1}^{(1)}-\mathbb{W}_{n+M^L-1}^{(1)} & \Cdots & \mathbb{D}_{n+M^L-1}^{(M^R)}-\mathbb{W}_{n+M^L-1}^{(M^R)}
        \end{bNiceMatrix} \\
        = - \begin{bNiceMatrix}
            \mathbb{B}_{n+M^L}^{(1)} & \Cdots & \mathbb{B}_{n+M^L}^{(M^L)} & \mathbb{D}_{n+M^L}^{(1)}-\mathbb{W}_{n+M^L}^{(1)} & \Cdots & \mathbb{D}_{n+M^L}^{(M^R)}-\mathbb{W}_{n+M^L}^{(M^R)}
        \end{bNiceMatrix}.        
    \end{multline*}
\end{Proposition}

\begin{proof}
    Applying a right eigenvector of $L(x)$ to Equation \eqref{Eq: ConnectB} and evaluating at the corresponding eigenvalue yields:
    \begin{equation*}
        \Omega B(\lambda_i) \vec{\boldsymbol{l}}_i = 0.
    \end{equation*}
    For $n \geq M^R$, this relation expands entrywise as:
    \begin{equation*}
        \Omega_{n,n-M^R}B_{n-M^R}(\lambda_i)\vec{\boldsymbol{l}}_i + \cdots + \Omega_{n,n+M^L-1}B_{n+M^L-1}(\lambda_i)\vec{\boldsymbol{l}}_i + B_{n+M^L}(\lambda_i)\vec{\boldsymbol{l}}_i = 0.
    \end{equation*}
    Using the notation introduced above, we obtain:
    \begin{equation*}
        \begin{bNiceMatrix}
            \Omega_{n,n-M^R} & \Cdots & \Omega_{n,n+M^L-1}
        \end{bNiceMatrix} 
        \begin{bNiceMatrix}
            \mathbb{B}_{n-M^R}^{(i)} \\ \Vdots \\\\ \mathbb{B}_{n+M^L-1}^{(i)}
        \end{bNiceMatrix} 
        = -\mathbb{B}_{n+M^L}^{(i)}.
    \end{equation*}
    Considering all $M^L$ roots of $\det L(x)$ provides $M^L$ equations for the $M^R+M^L$ unknowns.

    To complete the system, we examine Equation \eqref{Eq:ConnectD}. Direct evaluation at $x=\rho_i$ is not feasible, so we consider the limit:
    \begin{align*}
        \lim_{x \rightarrow \rho_i} \tilde{D}(x) R(x) \vec{\boldsymbol{r}}_i 
        &= \lim_{x \rightarrow \rho_i} \left( \int_\Delta \tilde{B}(y) \frac{\d \mu(y)}{x-y}R^{-1}(y)R(x) + \sum_{j=1}^{M^R} \tilde{B}(\rho_j) L(\rho_j)\boldsymbol{\xi}_j(\rho_j)\cev{\boldsymbol{r}}_i\frac{R(x)}{x-\rho_j} \right) \vec{\boldsymbol{r}}_i \\
        &= \tilde{B}(\rho_i) L(\rho_i) \boldsymbol{\xi}_i(\rho_i)\cev{\boldsymbol{r}}_i \lim_{x \rightarrow \rho_i} \frac{R(x)-R(\rho_i)}{x-\rho_i} \vec{\boldsymbol{r}}_i,
    \end{align*}
    where we have used $R(\rho_i)\vec{\boldsymbol{r}}_{i} = 0$. The remaining non-zero term becomes:
    \begin{equation*}
        \lim_{x \rightarrow \rho_i} \tilde{D}(x) R(x) \vec{\boldsymbol{r}}_i = \tilde{B}(\rho_i) L(\rho_i) \boldsymbol{\xi}(\rho_i)_i\cev{\boldsymbol{r}}_i R'(\rho_i) \vec{\boldsymbol{r}}_i.
    \end{equation*}
    
    For $n \geq M^R$, Lemma \ref{Lema: R-BivaMP} and the orthogonality relations imply:
    \begin{align*}
        \sum_{a=1}^p\sum_{b=1}^q \int_\Delta \tilde{B}_n^{(b)}(y) \d \tilde{\mu}_{b,a}(y) \frac{\left( R(x)-R(y) \right)_{a,\Tilde{a}}}{x-y} &= 0, \quad \Tilde{a}\in \{1,\dots, p\}.
    \end{align*}
    Evaluating at $z=\rho_i$ and applying the right eigenvector, Equation \eqref{Eq:ConnectD} becomes:
    \begin{multline*}
        \Omega_{n,n-M^R}\left(D_{n-M^R}(\rho_i)-B_{n-M^R}(\rho_i)\boldsymbol{\xi}_i(\rho_i)\cev{\boldsymbol{r}}_i R'(\rho_i)\right)\vec{\boldsymbol{r}}_i \\
        + \cdots + \Omega_{n,n+M^L-1}\left(D_{n+M^L-1}(\rho_i)-B_{n+M^L-1}(\rho_i)\boldsymbol{\xi}_i(\rho_i)\cev{\boldsymbol{r}}_i R'(\rho_i)\right)\vec{\boldsymbol{r}}_i \\
        + \left(D_{n+M^L}(\rho_i)-B_{n+M^L}(\rho_i)\boldsymbol{\xi}_i(\rho_i)\cev{\boldsymbol{r}}_i R'(\rho_i)\right)\vec{\boldsymbol{r}}_i = 0.
    \end{multline*}
    Using Definition \ref{Def:NotationBD}, this yields:
    \begin{equation*}
        \begin{bNiceMatrix}
            \Omega_{n,n-M^R} & \Cdots & \Omega_{n,n+M^L-1}
        \end{bNiceMatrix} 
        \begin{bNiceMatrix}
            \mathbb{D}_{n-M^R}^{(i)}-\mathbb{W}_{n-M^R}^{(i)} \\ \Vdots \\\\ \mathbb{D}_{n+M^L-1}^{(i)}-\mathbb{W}_{n+M^L-1}^{(i)}
        \end{bNiceMatrix} 
        = - \left( \mathbb{D}_{n+M^L}^{(i)}-\mathbb{W}_{n+M^L}^{(i)} \right).
    \end{equation*}
    Including all roots of $\det R(x)$ provides the remaining $M^R$ equations to complete the system.
\end{proof}
\begin{Proposition} \label{prop: Commutator}
    For $n \geq \mathrm{max}\{M^R,M^L\}$, the commutator $\left[ \Omega, \Pi_{n-1} \right]$ (where $\Pi_n$ is the diagonal projection matrix with ones in the first $n+1$ entries) satisfies: 
    \begin{align*}
   \NiceMatrixOptions{cell-space-top-limit=3pt} 
   \left[ \Omega, \Pi_{n-1} \right] = \left[\begin{NiceArray}{c|c}
	   0_{n\times n} & -\Omega^2 \\ \hline
	   \Omega^3 & 0_{\infty \times \infty}
    \end{NiceArray}\right], 
    \end{align*}
    where 
    \begin{align*}
        \Omega^2 & = \begin{bNiceMatrix}
            0 & \Cdots[shorten-end=7pt] \\
            \Vdots \\
            1 & 0 & \Cdots[shorten-end=7pt] \\
            \Omega_{n-M^L+1,n} & 1 & 0& \Cdots[shorten-end=7pt] \\
            \Vdots & \Ddots &  & \\
            \Omega_{n-1,n} & \Cdots & \Omega_{n-1,n+M^L-2} & 1 & 0 & \Cdots[shorten-end=7pt]
        \end{bNiceMatrix}, & \Omega^3 = \begin{bNiceMatrix}
            0 & \Cdots & 0 & \Omega_{n,n-M^R} & \Cdots & \Omega_{n,n-1} \\
            \Vdots & & & \Ddots & \Ddots & \Vdots \\
            0 & \Cdots & & & 0 & \Omega_{n+M^R-1,n-1} \\
            \Vdots & & & & & 0 \\
            & & & & & \Vdots[shorten-end=4pt] 
        \end{bNiceMatrix}. 
    \end{align*}
\end{Proposition}
\begin{Definition}
For $n \geq \mathrm{max}\{M^R,M^L\}$,	$a,d\in\{1,\dots,p\}$ and $b\in\{1,\dots,q\}$, we introduce the notation
	\[\begin{aligned}
	\tilde A_{n,M^L,M^R}^{(a)}(x)&\coloneq	\begin{bNiceMatrix}
		\tilde{A}_{n-M^L}^{(a)}(x) & \Cdots & \tilde{A}_{n-1}^{(a)}(x) & \tilde{A}_{n}^{(a)}(x) & \Cdots & \tilde{A}_{n+M^R-1}^{(a)}(x)
	\end{bNiceMatrix} \\
	\Omega_{n,M^R,M^L}&\coloneq \begin{bNiceMatrix}
		0 & \Cdots & & 0 & 1 & 0 & \Cdots & 0 \\
		\Vdots &  &&  \Vdots& \Omega_{n-M^L+1,n} & \Ddots & \Ddots& \Vdots \\
		& & & &  \Vdots& \Ddots[shorten-end=-30pt] & \Ddots & 0\\
		0 & \Cdots && 0 & \Omega_{n-1,n} & \Cdots & \Omega_{n-1,n+M^L-2} & 1 \\
		-\Omega_{n,n-M^R} & \Cdots && -\Omega_{n,n-1} & 0 & \Cdots & 0 & 0\\
		0 & \Ddots[shorten-end=-35pt] && \Vdots & \Vdots & & \Vdots & \Vdots \\
		\Vdots & \Ddots &&&&&&\\[-10pt]
		0 & \Cdots & 0 &-\Omega_{n+M^R-1,n-1} & 0 & \Cdots & 0 & 0 \\
	\end{bNiceMatrix}
\end{aligned}\]
and
\[
\begin{aligned}
	B_{n,M^R,M^L}^{(b)}(y)&\coloneq\begin{bNiceMatrix}
		B_{n-M^R}^{(b)}(y) \\
		\Vdots \\\\
		B_{n-1}^{(b)}(y) \\[2pt]
		B_{n}^{(b)}(y) \\
		\Vdots \\\\
		B_{n+M^L-1}^{(b)}(y) \\
	\end{bNiceMatrix},&
	D_{n,M^R,M^L}^{(d)}(y)&\coloneq\begin{bNiceMatrix}
		D_{n-M^R}^{(d)}(y) \\
		\Vdots \\\\
		D_{n-1}^{(d)}(y) \\[2pt]
		D_{n}^{(d)}(y) \\
		\Vdots \\\\
		D_{n+M^L-1}^{(d)}(y) \\
	\end{bNiceMatrix}.
	\end{aligned}\]
	We will also use
	\begin{align*}
	\tilde A_{n,M^L,M^R}(x)=\begin{bNiceMatrix}
		\tilde A_{n,M^L,M^R}^{(1)}(x)\\
		\Vdots\\\\\\
		\tilde A_{n,M^L,M^R}^{(p)}(x)
	\end{bNiceMatrix}\in \C^{p\times(M^R+M^L)}(x)
	\end{align*}
\end{Definition}
 \begin{Proposition}
 The products $\tilde{A}(x)\left[ \Omega,\Pi_{n-1} \right]B(y)$ and $\tilde{A}(x)\left[ \Omega,\Pi_{n-1} \right]D(y)$ yield matrices of sizes $p\times q$ and $p \times p$, respectively. 
 
 Moreover, for $n \geq \mathrm{max}\{M^R,M^L\}$, $a,d\in\{1,\dots,p\}$ and $b\in\{1,\dots,q\}$, these products expand entrywise as:
    \begin{align*}
        \tilde{A}^{(a)}(x)\left[ \Omega,\Pi_{n-1} \right]B^{(b)}(y) &= - \tilde A^{(a)}_{n,M^L,M^R}(x)\Omega_{n,M^R,M^L}B^{(b)}_{n,M^L,M^R}(y),\\
      \tilde{A}^{(a)}(x)\left[ \Omega,\Pi_{n-1} \right]D^{(d)}(y) &= - \tilde A^{(a)}_{n,M^L,M^R}(x)\Omega_{n,M^R,M^L}D^{(d)}_{n,M^L,M^R}(y).
    \end{align*}
\end{Proposition}
\begin{Remark} \label{Rm: CommutatorSpecial}
    These formulas remain valid for $n < M^L$ or $n<M^R$. When $n < M^R$, $\Omega^3$ begins at the first column with non-negative integer components of $\Omega$. For $n = M^L-1$, $\Omega^2$ starts from the row containing the first non-zero $\Omega$ component. For $n < M^L-1$, $\Omega^2$ initiates at the first row with non-negative integer $\Omega$ components. 
\end{Remark}

Although we have not proven these results concretely, we examine an explicit example. 
\begin{Example}
    Let's consider the case where $M^L=3$ and $M^R=2$. The connection matrix takes the form: 
    \begin{equation*}
        \Omega = \begin{bNiceMatrix}
            \Omega_{0,0} & \Omega_{0,1} & \Omega_{0,2} & 1 & 0 & \Cdots[shorten-end=-20pt] &  &&\phantom{text}\\[-2pt]
            \Omega_{1,0} &     \Omega_{1,1}  &     \Omega_{1,2} & \Omega_{1,3} & 1 & \Ddots[shorten-end=-24pt] &&& \phantom{text}\\
            \Omega_{2,0} &  \Omega_{2,1}&  \Omega_{2,2} & \Omega_{2,3} & \Omega_{2,4} & 1  &&&\phantom{text} \\
            0 & \Omega_{3,1} & \Omega_{3,2} & \Omega_{3,3}  & \Omega_{3,4} & \Omega_{3,5} & 1 &  &  \phantom{text}\\
            \Vdots[shorten-end=-10pt] & \Ddots[shorten-end=0pt] & \Ddots[shorten-end=-5pt] &\Ddots[shorten-end=-5pt] &\Ddots[shorten-end=-5pt] & \Ddots&  \Ddots[shorten-end=0pt] & \Ddots[shorten-end=0pt] &  \phantom{text}
            \\&&&&&&&&\\&&&&&&&&\\&&&&&&&&\\
            \phantom{text}&  \phantom{text}&  \phantom{text}&  \phantom{text}&  \phantom{text}&  \phantom{text} & \phantom{text}&\phantom{text}&\phantom{text}
        \end{bNiceMatrix}.
    \end{equation*}
For $n=1$, we find: 
\[\begin{aligned}
    \Omega^3 & = \begin{bNiceMatrix}
        \Omega_{1,0} \\ \Omega_{2,0} \\ 0 \\ \Vdots[shorten-end=3pt]\\\\
    \end{bNiceMatrix}, & \Omega^2 = & \begin{bNiceMatrix}
        \Omega_{0,1} & \Omega_{0,2} & 1 & 0 & \Cdots[shorten-end=6pt] & &
    \end{bNiceMatrix}.
\end{aligned}\]
For $n=2$, we have:
\[\begin{aligned}
    \Omega^3 & = \begin{bNiceMatrix}
        \Omega_{2,0} & \Omega_{2,1} \\
        0 & \Omega_{3,1} \\
    0   & 0 \\
     \Vdots[shorten-end=3pt]     & \Vdots[shorten-end=3pt] \\ \\
    \end{bNiceMatrix}, & \Omega^2 = & \begin{bNiceMatrix}
        \Omega_{0,2} &  1 & 0 & 0&\Cdots[shorten-end=8pt] &&\\
        \Omega_{1,2} & \Omega_{1,3} & 1 & 0 & \Cdots[shorten-end=8pt]&&
    \end{bNiceMatrix}.
\end{aligned}\]
Lastly, for $n=3$: 
\[\begin{aligned}
    \Omega^3 & = \begin{bNiceMatrix}
        0 & \Omega_{3,1} & \Omega_{3,2} \\
        0& 0 & \Omega_{4,2} \\
0& 0& 0 \\
          \Vdots[shorten-end=3pt]        & \Vdots[shorten-end=3pt] & \Vdots[shorten-end=3pt]
  \\\\  \end{bNiceMatrix}, & \Omega^2 = & \begin{bNiceMatrix}
        1 & 0 & 0 & 0&\Cdots[shorten-end=8pt] && \\
        \Omega_{1,3} &  1 & 0 & 0&\Cdots[shorten-end=8pt] &&\\
        \Omega_{2,3} & \Omega_{2,4} & 1 & 0 & \Cdots[shorten-end=8pt] &&
    \end{bNiceMatrix}.
\end{aligned}\]
For $n=2$, entrywise, the product $\tilde{A}(x)\left[ \Omega,\Pi_{n-1} \right]B(y)$ reads: 
\begin{multline*}
    \tilde{A}^{(a)}(x)\left[ \Omega,\Pi_{n-1} \right]B^{(b)}(y)\\ = -\begin{bNiceMatrix}
        \tilde{A}_0^{(a)}(x) &   \tilde{A}_1^{(a)}(x) & \tilde{A}_2^{(a)}(x) & \tilde{A}_3^{(a)}(x)
    \end{bNiceMatrix}\begin{bNiceMatrix}
        0 & 0 & \Omega_{0,2} &  1 & 0 \\
        0 & 0 & \Omega_{1,2} & \Omega_{1,3} & 1 \\
        -\Omega_{2,0} & -\Omega_{2,1} & 0 & 0 & 0\\
        0 & -\Omega_{3,1} & 0 & 0 & 0
    \end{bNiceMatrix}\begin{bNiceMatrix}
        B_0^{(b)}(y) \\[2pt] B_1^{(b)}(y) \\[2pt] B_2^{(b)}(y)\\[2pt] B_3^{(b)}(y)\\[2pt]  B_4^{(b)}(y)
    \end{bNiceMatrix}.
\end{multline*}
\end{Example}

\begin{Proposition} \label{Prop: KernelConex}
    The connection formulas for the kernel polynomials are:
    \begin{equation} \label{Eq: KernelConex}
        \tilde{K}^{[n-1]}(x,y)L(y) = R(x)K^{[n-1]}(x,y) - \tilde{A}(x)\left[ \Omega,\Pi_{n-1} \right]B(y), 
    \end{equation}
    and for $n \geq M^R$, the mixed-type CD kernels satisfy: 
    \begin{equation} \label{Eq: MixedKernelConex}
        \tilde{K}^{[n-1]}_D(x,y)R(y) = R(x)K^{[n-1]}_D(x,y) - \tilde{A}(x)\left[ \Omega,\Pi_{n-1} \right]D(y) + \frac{R(x)-R(y)}{x-y}.
    \end{equation}
\end{Proposition}
\begin{proof}
    The result follows from Proposition \ref{prop: Commutator} and Corollary \ref{Col: ProycK_2}. For additional details, we refer to Proposition 3.6 in \cite{Manas_Rojas_Christoffel} and Proposition 2.9 in \cite{Manas_Rojas_Geronimus}, noting that these references employ slightly different kernel definitions.
\end{proof}
\begin{Definition} \label{Def: NotationK}
    Let us introduce the following functions taking values in $\C^p$:
    \begin{align*}
        \mathbb{K}^{[n-1],(j)}(x) & \coloneq K^{[n-1]}(x,\lambda_j)\vec{\boldsymbol{l}}_j, & \mathbb{K}_D^{[n-1],(i)}(x) &\coloneq K^{[n-1]}_D(x,\rho_i)\vec{\boldsymbol{r}}_i-K^{[n-1]}(x,\rho_i)\boldsymbol{\xi}_i(\rho_i)\cev{\boldsymbol{r}}_iR'(\rho_i)\vec{\boldsymbol{r}}_i,
    \end{align*}
    for $n \in \mathbb{N}$, $i \in \{1,\dots,M^R \}$ and $j \in \{ 1,\dots,M^L \}$.
\end{Definition}
\begin{Proposition} \label{Prop: Kfinal}
    The following formula holds for $n\geq M^R$: 
    \begin{multline*}
        R(x)\begin{bNiceMatrix}
            \mathbb{K}^{[n-1],(1)}(x) & \Cdots & \mathbb{K}^{[n-1],(M^L)}(x) & \mathbb{K}^{[n-1],(1)}_D(x) + \frac{\vec{\boldsymbol{r}}_1}{x-\rho_1} & \Cdots & \mathbb{K}^{[n-1],(M^R)}_D(x) + \frac{\vec{\boldsymbol{r}}_{M^R}}{x-\rho_{M^R}}
        \end{bNiceMatrix} \\ = \tilde{A}_{n,M^L,M^R}(x) \Omega_{n,M^R,M^L} \begin{bNiceMatrix}
            \mathbb{B}_{n-M^R}^{(1)} & \Cdots & \mathbb{B}_{n-M^R}^{(M^L)} & \mathbb{D}_{n-M^R}^{(1)}-\mathbb{W}_{n-M^R}^{(1)} & \Cdots & \mathbb{D}_{n-M^R}^{(M^R)}-\mathbb{W}_{n-M^R}^{(M^R)} \\ \Vdots & & \Vdots & \Vdots & & \Vdots \\ \\\\
            \mathbb{B}_{n+M^L-1}^{(1)} & \Cdots & \mathbb{B}_{n+M^L-1}^{(M^L)} & \mathbb{D}_{n+M^L-1}^{(1)}-\mathbb{W}_{n+M^L-1}^{(1)} & \Cdots & \mathbb{D}_{n+M^L-1}^{(M^R)}-\mathbb{W}_{n+M^L-1}^{(M^R)}
        \end{bNiceMatrix}.
    \end{multline*}
\end{Proposition}
\begin{proof}
    Applying a right eigenvector of $L(x)$ to Equation \eqref{Eq: KernelConex} and evaluating at the corresponding eigenvalue yields:
    \begin{equation*}
        \tilde{K}^{[n-1]}(x,\lambda_i)L(\lambda_i)\vec{\boldsymbol{l}}_i = R(x) K^{[n-1]}(x,\lambda_i)\vec{\boldsymbol{l}}_i - \tilde{A}(x)\left[ \Omega,\Pi_{n-1} \right]B(\lambda_i)\vec{\boldsymbol{l}}_i.
    \end{equation*}
    The left hand side vanishes identically. Using Definitions \ref{Def:NotationBD} and \ref{Def: NotationK} the equation can be rewritten as:
    \begin{equation} \label{Eq: Kevaluate1}
        R(x)\mathbb{K}^{[n-1],(i)}(x) = \tilde{A}_{n,M^L,M^R}(x) \Omega_{n,M^R,M^L}\begin{bNiceMatrix}
            \mathbb{B}_{n-M^R}^{(i)} \\
            \Vdots \\
            \mathbb{B}_{n-1}^{(i)} \\[4pt]
            \mathbb{B}_{n}^{(i)} \\
            \Vdots \\
            \mathbb{B}_{n+M^L-1}^{(i)} \\
        \end{bNiceMatrix}.
    \end{equation}
    For Equation \eqref{Eq: MixedKernelConex}, we first establish the limit: 
    \begin{equation*}
        \lim_{y\to\rho_i} \tilde{K}^{[n-1]}_D(x,y)R(y)\vec{\boldsymbol{r}}_i = R(x)K^{[n-1]}(x,\rho_i)\boldsymbol{\xi}_i(\rho_i)\cev{\boldsymbol{r}}_iR'(\rho_i)\vec{\boldsymbol{r}}_i-\tilde{A}(x)\left[ \Omega,\Pi_{n-1} \right]B(\rho_i)\boldsymbol{\xi}_i(\rho_i)\cev{\boldsymbol{r}}_iR'(\rho_i)\vec{\boldsymbol{r}}_i,
    \end{equation*}
    The prove of the limit is similar to the limit carried out in Proposition \ref{Prop: OmegaComp} (see Proposition 2.11. in \cite{Manas_Rojas_Geronimus} for further details). Evaluating Equation \eqref{Eq: MixedKernelConex} at $z=\rho_i$ with the right eigenvector gives:   
    \begin{multline*}
        -R(x)\left( K_D^{[n-1]}(x,\rho_i)\vec{\boldsymbol{r}}_i - K^{[n-1]}(x,\rho_i)\boldsymbol{\xi}_i(\rho_i)\cev{\boldsymbol{r}}_iR'(\rho_i)\vec{\boldsymbol{r}}_i + \frac{\vec{\boldsymbol{r}}_i}{x-\rho_i}\right) \\ = \tilde{A}(x)\left[ \Omega,\Pi_{n-1} \right]\left( D(\rho_i)\vec{\boldsymbol{r}}_i - B(\rho_i)\boldsymbol{\xi}_i(\rho_i)\cev{\boldsymbol{r}}_iR'(\rho_i)\vec{\boldsymbol{r}}_i \right).
    \end{multline*}
    Which can be rewritten as: 
    \begin{equation} \label{Eq: Kevaluate2}
        R(x)\left( \mathbb{K}_D^{[n-1]}(x) + \frac{\vec{\boldsymbol{r}}_i}{x-\rho_i} \right) = \tilde{A}_{n,M^L,M^R}(x) \Omega_{n,M^R,M^L}\begin{bNiceMatrix}
            \mathbb{D}_{n-M^R}^{(i)}-\mathbb{W}_{n-M^R}^{(i)} \\
            \Vdots \\\\\\
            \mathbb{D}_{n+M^L-1}^{(i)}-\mathbb{W}_{n+M^L-1}^{(i)}
        \end{bNiceMatrix}.
    \end{equation}
    Combining equations \eqref{Eq: Kevaluate1} and \eqref{Eq: Kevaluate2} for all eigenvalues of $L(x)$ and $R(x)$ yields the desired matrix relation.
\end{proof}
Let us now introduce a key element in our study, the $\tau$-determinants. 
\begin{Definition} \label{Def: Taudeter}
For $n \geq M^R$,   the $\tau$-determinants are defined as follows: 
    \begin{equation*}
        \tau_n \coloneqq \begin{vNiceMatrix}
            \mathbb{B}_{n-M^R}^{(1)} & \Cdots & \mathbb{B}_{n-M^R}^{(M^L)} & \mathbb{D}_{n-M^R}^{(1)}-\mathbb{W}_{n-M^R}^{(1)} & \Cdots & \mathbb{D}_{n-M^R}^{(M^R)}-\mathbb{W}_{n-M^R}^{(M^R)} \\ \Vdots & & \Vdots & \Vdots & & \Vdots \\ \\\\
            \mathbb{B}_{n+M^L-1}^{(1)} & \Cdots & \mathbb{B}_{n+M^L-1}^{(M^L)} & \mathbb{D}_{n+M^L-1}^{(1)}-\mathbb{W}_{n+M^L-1}^{(1)} & \Cdots & \mathbb{D}_{n+M^L-1}^{(M^R)}-\mathbb{W}_{n+M^L-1}^{(M^R)}
        \end{vNiceMatrix}.
    \end{equation*}
\end{Definition}
\begin{Theorem} \label{Th: ChrisFormulas}
    The following Christoffel type formulas hold for $n \geq M^R$
    \begin{multline} \label{Eq: ChrisFormulasA}
        \tilde{A}^{(a)}_{n-1}(x)
        =\frac{1}{\tau_n} \sum_{\tilde{a}=1}^p R_{a,\tilde{a}}(x) \\ \times \begin{vNiceMatrix}
        \mathbb{B}_{n-M^R}^{(1)} & \Cdots & \mathbb{B}_{n-M^R}^{(M^L)} & \mathbb{D}_{n-M^R}^{(1)}-\mathbb{W}_{n-M^R}^{(1)} & \Cdots & \mathbb{D}_{n-M^R}^{(M^R)}-\mathbb{W}_{n-M^R}^{(M^R)} \\ 
        \Vdots & & \Vdots & \Vdots & & \Vdots \\ \\\\
        \mathbb{B}_{n+M^L-2}^{(1)} & \Cdots & \mathbb{B}_{n+M^L-2}^{(M^L)} & \mathbb{D}_{n+M^L-2}^{(1)}-\mathbb{W}_{n+M^L-2}^{(1)} & \Cdots & \mathbb{D}_{n+M^L-2}^{(M^R)}-\mathbb{W}_{n+M^L-2}^{(M^R)}\\[5pt]
        \mathbb{K}_{\tilde{a}}^{[n-1-M^R],(1)}(x) & \Cdots & \mathbb{K}_{\tilde{a}}^{[n-1-M^R],(M^L)}(x) & \mathbb{K}_{D,\tilde{a}}^{[n-1-M^R],(1)}(x) + \frac{\vec{r}_{1,\tilde{a}}}{x-\rho_1} & \Cdots & \mathbb{K}_{D,\tilde{a}}^{[n-1-M^R],(M^R)}(x) + \frac{\vec{r}_{M^R,\tilde{a}}}{x-\rho_{M^R}} 
    \end{vNiceMatrix}, 
    \end{multline}
and
    \begin{equation} \label{Eq: ChrisFormulasB}
        \tilde{B}_n^{(b)}(x) = \frac{1}{\tau_n}\sum_{\tilde{b}=1}^q \begin{vNiceMatrix}
            \mathbb{B}_{n-M^R}^{(1)} & \Cdots & \mathbb{B}_{n-M^R}^{(M^L)} & \mathbb{D}_{n-M^R}^{(1)}-\mathbb{W}_{n-M^R}^{(1)} & \Cdots & \mathbb{D}_{n-M^R}^{(M^R)}-\mathbb{W}_{n-M^R}^{(M^R)} & B_{n-M^R}^{(\tilde{b})}(x) \\ 
            \Vdots & & \Vdots & \Vdots & & \Vdots & \Vdots \\ \\\\
             \mathbb{B}_{n+M^L}^{(1)} & \Cdots & \mathbb{B}_{n+M^L}^{(M^L)} & \mathbb{D}_{n+M^L}^{(1)}-\mathbb{W}_{n+M^L}^{(1)} & \Cdots & \mathbb{D}_{n+M^L}^{(M^R)}-\mathbb{W}_{n+M^L}^{(M^R)} & B_{n+M^L}^{(\tilde{b})}(x) 
        \end{vNiceMatrix} L_{\tilde{b},b}^{-1}(x). 
    \end{equation}
\end{Theorem}
\begin{proof}
    Since it has been assumed that $\tau_n \neq 0$ for $n\geq M^R$, Proposition \ref{Prop: Kfinal} can be rearranged as:
    \begin{multline*}
        -R(x)\begin{bNiceMatrix}
            \mathbb{K}^{[n-1],(1)}(x) & \Cdots & \mathbb{K}^{[n-1],(M^L)}(x) & \mathbb{K}^{[n-1],(1)}_D(x) + \frac{\vec{\boldsymbol{r}}_{1}}{x-\rho_1} & \Cdots & \mathbb{K}^{[n-1],(M^R)}_D(x) + \frac{\vec{\boldsymbol{r}}_{M^R}}{x-\rho_{M^R}}
        \end{bNiceMatrix} \\ \times \begin{bNiceMatrix}
            \mathbb{B}_{n-M^R}^{(1)} & \Cdots & \mathbb{B}_{n-M^R}^{(M^L)} & \mathbb{D}_{n-M^R}^{(1)}-\mathbb{W}_{n-M^R}^{(1)} & \Cdots & \mathbb{D}_{n-M^R}^{(M^R)}-\mathbb{W}_{n-M^R}^{(M^R)} \\ \Vdots & & \Vdots & \Vdots & & \Vdots \\ \\\\
            \mathbb{B}_{n+M^L-1}^{(1)} & \Cdots & \mathbb{B}_{n+M^L-1}^{(M^L)} & \mathbb{D}_{n+M^L-1}^{(1)}-\mathbb{W}_{n+M^L-1}^{(1)} & \Cdots & \mathbb{D}_{n+M^L-1}^{(M^R)}-\mathbb{W}_{n+M^L-1}^{(M^R)}
        \end{bNiceMatrix}^{-1}= - \tilde{A}_{n,M^L,M^R}(x)\Omega_{n,M^R,M^L} .
    \end{multline*}
    Let us multiply the last equation by the vector $\begin{bNiceMatrix}
        0 & \Cdots & 0 & 1
    \end{bNiceMatrix}^\top\in{\C^{M^L+M^R}}$, yielding: 
    \begin{equation*}
  \tilde{A}_{n,M^L,M^R}(x)\Omega_{n,M^R,M^L} \begin{bNiceMatrix}
        0 \\ \Vdots \\\\\\ 0 \\ 1
    \end{bNiceMatrix} = \begin{bNiceMatrix}
        \tilde{A}^{(1)}_{n-1}(x) \\
        \Vdots \\\\\\
        \tilde{A}^{(p)}_{n-1}(x)
    \end{bNiceMatrix}. 
    \end{equation*}

    Which leads to:
    \begin{multline*}
        -\begin{bNiceMatrix}
        \tilde{A}^{(1)}_{n-1}(x) \\
        \Vdots \\\\\\
        \tilde{A}^{(p)}_{n-1}(x)
    \end{bNiceMatrix}\\ = -R(x)\begin{bNiceMatrix}
            \mathbb{K}^{[n-1],(1)}(x) & \Cdots & \mathbb{K}^{[n-1],(M^L)}(x) & \mathbb{K}^{[n-1],(1)}_D(x) + \frac{\cev{\boldsymbol{r}}_{1}}{x-\rho_1} & \Cdots & \mathbb{K}^{[n-1],(M^R)}_D(x) + \frac{\cev{\boldsymbol{r}}_{M^R}}{x-\rho_{M^R}}
        \end{bNiceMatrix} \\
        \times\begin{bNiceMatrix}
            \mathbb{B}_{n-M^R}^{(1)} & \Cdots & \mathbb{B}_{n-M^R}^{(M^L)} & \mathbb{D}_{n-M^R}^{(1)}-\mathbb{W}_{n-M^R}^{(1)} & \Cdots & \mathbb{D}_{n-M^R}^{(M^R)}-\mathbb{W}_{n-M^R}^{(M^R)} \\ \Vdots & & \Vdots & \Vdots & & \Vdots \\ \\\\
            \mathbb{B}_{n+M^L-1}^{(1)} & \Cdots & \mathbb{B}_{n+M^L-1}^{(M^L)} & \mathbb{D}_{n+M^L-1}^{(1)}-\mathbb{W}_{n+M^L-1}^{(1)} & \Cdots & \mathbb{D}_{n+M^L-1}^{(M^R)}-\mathbb{W}_{n+M^L-1}^{(M^R)}
        \end{bNiceMatrix}^{-1}\begin{bNiceMatrix}
        0 \\ \Vdots \\\\\\ 0 \\ 1
    \end{bNiceMatrix}. 
    \end{multline*}
    Entrywise, we find: 
    \begin{multline*}
        -\tilde{A}_{n-1}^{(a)}(x) = \frac{1}{\tau_n}\sum_{\tilde{a}=1}^pR_{a,\tilde{a}}(x) \\
        \times \begin{vNiceMatrix}
        \mathbb{B}_{n-M^R}^{(1)} & \Cdots & \mathbb{B}_{n-M^R}^{(M^L)} & \mathbb{D}_{n-M^R}^{(1)}-\mathbb{W}_{n-M^R}^{(1)} & \Cdots & \mathbb{D}_{n-M^R}^{(M^R)}-\mathbb{W}_{n-M^R}^{(M^R)} & 0 \\ 
        \Vdots & & \Vdots & \Vdots & & \Vdots & \Vdots \\ \\\\
        \mathbb{B}_{n+M^L-1}^{(1)} & \Cdots & \mathbb{B}_{n+M^L-1}^{(M^L)} & \mathbb{D}_{n+M^L-1}^{(1)}-\mathbb{W}_{n+M^L-1}^{(1)} & \Cdots & \mathbb{D}_{n+M^L-1}^{(M^R)}-\mathbb{W}_{n+M^L-1}^{(M^R)} & 1 \\[5pt]
        \mathbb{K}_{\tilde{a}}^{[n-1],(1)}(x) & \Cdots & \mathbb{K}_{\tilde{a}}^{[n-1],(M^L)}(x) & \mathbb{K}_{D,\tilde{a}}^{[n-1],(1)}(x) + \frac{\vec{r}_{1,\tilde{a}}}{x-\rho_1} & \Cdots & \mathbb{K}_{D,\tilde{a}}^{[n-1],(M^R)}(x) + \frac{\vec{r}_{M^R,\tilde{a}}}{x-\rho_{M^R}}  & 0 
    \end{vNiceMatrix}.
    \end{multline*}
    Using Laplace's expansion for the determinant yields:
    \begin{multline*}
        \tilde{A}_{n-1}^{(a)}(x) = \frac{1}{\tau_n}\sum_{\tilde{a}=1}^pR_{a,\tilde{a}}(x) \\\times  \begin{vNiceMatrix}
            \mathbb{B}_{n-M^R}^{(1)} & \Cdots & \mathbb{B}_{n-M^R}^{(M^L)} & \mathbb{D}_{n-M^R}^{(1)}-\mathbb{W}_{n-M^R}^{(1)} & \Cdots & \mathbb{D}_{n-M^R}^{(M^R)}-\mathbb{W}_{n-M^R}^{(M^R)} \\ 
            \Vdots & & \Vdots & \Vdots & & \Vdots \\ \\\\
            \mathbb{B}_{n+M^L-2}^{(1)} & \Cdots & \mathbb{B}_{n+M^L-2}^{(M^L)} & \mathbb{D}_{n+M^L-2}^{(1)}-\mathbb{W}_{n+M^L-2}^{(1)} & \Cdots & \mathbb{D}_{n+M^L-2}^{(M^R)}-\mathbb{W}_{n+M^L-2}^{(M^R)} \\[5pt]
            \mathbb{K}_{\tilde{a}}^{[n-1],(1)}(x) & \Cdots & \mathbb{K}_{\tilde{a}}^{[n-1],(M^L)}(x) & \mathbb{K}_{D,\tilde{a}}^{[n-1],(1)}(x) + \frac{\vec{r}_{1,\tilde{a}}}{x-\rho_1} & \Cdots & \mathbb{K}_{D,\tilde{a}}^{[n-1],(M^R)}(x) + \frac{\vec{r}_{M^R,\tilde{a}}}{x-\rho_{M^R}}  
        \end{vNiceMatrix}. 
    \end{multline*}
    The determinant can be expanded as: 
    \begin{multline*}
        \tilde{A}_{n-1}^{(a)}(x) = \frac{1}{\tau_n}\sum_{\tilde{a}=1}^pR_{a,\tilde{a}}(x) \\ 
        \times \left( \sum_{i=0}^{n-1} A_i(x)\begin{vNiceMatrix}
        \mathbb{B}_{n-M^R}^{(1)} & \Cdots & \mathbb{B}_{n-M^R}^{(M^L)} & \mathbb{D}_{n-M^R}^{(1)}-\mathbb{W}_{n-M^R}^{(1)} & \Cdots & \mathbb{D}_{n-M^R}^{(M^R)}-\mathbb{W}_{n-M^R}^{(M^R)} \\ 
        \Vdots & & \Vdots & \Vdots & & \Vdots \\ \\
        \mathbb{B}_{n+M^L-2}^{(1)} & \Cdots & \mathbb{B}_{n+M^L-2}^{(M^L)} & \mathbb{D}_{n+M^L-2}^{(1)}-\mathbb{W}_{n+M^L-2}^{(1)} & \Cdots & \mathbb{D}_{n+M^L-2}^{(M^R)}-\mathbb{W}_{n+M^L-2}^{(M^R)} \\[5pt]
        \mathbb{B}_{i}^{(1)} & \Cdots & \mathbb{B}_{i}^{(M^L)} & \mathbb{D}_{i}^{(1)}-\mathbb{W}_{i}^{(1)} & \Cdots & \mathbb{D}_{i}^{(M^R)}-\mathbb{W}_{i}^{(M^R)} 
    \end{vNiceMatrix} \right. \\[3pt]
    \left. + \begin{vNiceMatrix}
        \mathbb{B}_{n-M^R}^{(1)} & \Cdots & \mathbb{B}_{n-M^R}^{(M^L)} & \mathbb{D}_{n-M^R}^{(1)}-\mathbb{W}_{n-M^R}^{(1)} & \Cdots & \mathbb{D}_{n-M^R}^{(M^R)}-\mathbb{W}_{n-M^R}^{(M^R)} \\ 
        \Vdots & & \Vdots & \Vdots & & \Vdots \\ \\
        \mathbb{B}_{n+M^L-2}^{(1)} & \Cdots & \mathbb{B}_{n+M^L-2}^{(M^L)} & \mathbb{D}_{n+M^L-2}^{(1)}-\mathbb{W}_{n+M^L-2}^{(1)} & \Cdots & \mathbb{D}_{n+M^L-2}^{(M^R)}-\mathbb{W}_{n+M^L-2}^{(M^R)} \\[5pt]
        0 & \Cdots & 0 & \frac{\vec{\boldsymbol{r}}_{1}}{x-\rho_1} & \Cdots & \frac{\vec{\boldsymbol{r}}_{M^R}}{x-\rho_{M^R}} 
    \end{vNiceMatrix} \right). 
    \end{multline*}
    The sum truncates when $i \geq n-M^R$ and the desired relation is immediate. 
    
    For the $B(x)$ polynomials, we analyze Equation \eqref{Eq: ConnectB} entrywise:
    \begin{equation*}
        B_{n+M^L}^{(b)}(x) + \begin{bNiceMatrix}
            \Omega_{n,n-M^R} & \Cdots & \Omega_{n,n+M^L-1}
        \end{bNiceMatrix}\begin{bNiceMatrix}
            B_{n-M^R}^{(b)}(x) \\
            \Vdots \\\\\\
            B_{n+M^L-1}^{(b)}(x)
        \end{bNiceMatrix} = \sum_{\tilde{b}=1}^q \tilde{B}_{n}^{(\tilde{b})}(x)L_{\tilde{b},b}(x).
    \end{equation*}
    Using Proposition \ref{Prop: OmegaComp}, the equation can be rewritten as:
    \begin{multline*}
        B_{n+M^L}^{(b)}(x) - \begin{bNiceMatrix}
            \mathbb{B}_{n+M^L}^{(1)} & \Cdots & \mathbb{B}_{n+M^L}^{(M^L)} & \mathbb{D}_{n+M^L}^{(1)}-\mathbb{W}_{n+M^L}^{(1)} & \Cdots & \mathbb{D}_{n+M^L}^{(M^R)}-\mathbb{W}_{n+M^L}^{(M^R)}
        \end{bNiceMatrix} \\ \times \begin{bNiceMatrix}
            \mathbb{B}_{n-M^R}^{(1)} & \Cdots & \mathbb{B}_{n-M^R}^{(M^L)} & \mathbb{D}_{n-M^R}^{(1)}-\mathbb{W}_{n-M^R}^{(1)} & \Cdots & \mathbb{D}_{n-M^R}^{(M^R)}-\mathbb{W}_{n-M^R}^{(M^R)} \\ \Vdots & & \Vdots & \Vdots & & \Vdots \\ \\\\
            \mathbb{B}_{n+M^L-1}^{(1)} & \Cdots & \mathbb{B}_{n+M^L-1}^{(M^L)} & \mathbb{D}_{n+M^L-1}^{(1)}-\mathbb{W}_{n+M^L-1}^{(1)} & \Cdots & \mathbb{D}_{n+M^L-1}^{(M^R)}-\mathbb{W}_{n+M^L-1}^{(M^R)}
        \end{bNiceMatrix}^{-1} \begin{bNiceMatrix}
            B_{n-M^R}^{(b)}(x) \\
            \Vdots \\\\\\
            B_{n+M^L-1}^{(b)}(x)
        \end{bNiceMatrix} \\ = \sum_{\tilde{b}=1}^q \tilde{B}_{n}^{(\tilde{b})}(x)L_{\tilde{b},b}(x). 
    \end{multline*}
    Equivalently, 
    \begin{multline*}
        \sum_{\tilde{b}=1}^q \tilde{B}_{n}^{(\tilde{b})}(x)L_{\tilde{b},b}(x) = \frac{1}{\tau_n}\begin{bNiceMatrix}
            \mathbb{B}_{n-M^R}^{(1)} & \Cdots & \mathbb{B}_{n-M^R}^{(M^L)} & \mathbb{D}_{n-M^R}^{(1)}-\mathbb{W}_{n-M^R}^{(1)} & \Cdots & \mathbb{D}_{n-M^R}^{(M^R)}-\mathbb{W}_{n-M^R}^{(M^R)} & B_{n-M^R}^{(b)}(x) \\ 
            \Vdots & & \Vdots & \Vdots & & \Vdots & \Vdots \\ \\\\
             \mathbb{B}_{n+M^L}^{(1)} & \Cdots & \mathbb{B}_{n+M^L}^{(M^L)} & \mathbb{D}_{n+M^L}^{(1)}-\mathbb{W}_{n+M^L}^{(1)} & \Cdots & \mathbb{D}_{n+M^L}^{(M^R)}-\mathbb{W}_{n+M^L}^{(M^R)} & B_{n+M^L}^{(b)}(x) 
        \end{bNiceMatrix}. 
    \end{multline*}
    from which Equation \eqref{Eq: ChrisFormulasB} follows by inverting $L(x)$. 
\end{proof}
\begin{Remark} \label{Rm: n<M^R}
    Let's study in more detail Proposition \ref{Th: ChrisFormulas} when $n < M^R$. From Equation \eqref{Eq: ConnectB}, we have: 
    \begin{equation*}
        \Omega_{n,0}B_0(x) + \dots + \Omega_{0,M^L-1+n}B_{M^L-1+n}(x) = \tilde{B}_0(x)L(x). 
    \end{equation*}
    There are $M^L+n$ unknowns, corresponding to the $\Omega$ entries. However, the method from Proposition \ref{Prop: OmegaComp} can only determine $M^L$ equations. To see this, recall that the integral term in Equation \eqref{Eq:ConnectD} does not vanish for $n<M^R$. If we act with an eigenvector of $R(x)$ and take the limit $x \to \rho_i$ on the integral term, we find: 
    \begin{equation*}
        \lim_{x\to\rho_i} \left( \int_\Delta \tilde{B}(t)\d \tilde{\mu}(t) \frac{R(x)-R(t)}{x-t} \vec{\boldsymbol{r}}_i \right)^{[n]} = - \lim_{x\to\rho_i} \left( \int_\Delta \tilde{B}(t)\d \tilde{\mu}(t) \frac{R(t)}{x-t} \vec{\boldsymbol{r}}_i \right)^{[n]}. 
    \end{equation*}
    Explicitly decomposing the perturbed measure, we have: 
    \begin{multline*}
        \lim_{x\to\rho_i} \left( \int_\Delta \tilde{B}(t)\d \tilde{\mu}(t) \frac{R(x)-R(t)}{x-t} \vec{\boldsymbol{r}}_i \right)^{[n]} = - \left( \Omega \int_\Delta B(t) \frac{\d \mu (t)}{\rho_i - t} \vec{\boldsymbol{r}}_i \right. \\ \left. + \lim_{x \rightarrow \rho_i }\sum_{j=1}^{M^R} \tilde{B}(t)L(t)\boldsymbol{\xi}_j(t) \cev{\boldsymbol{r}}_j \frac{R(t)}{\rho_i-t}\vec{\boldsymbol{r}}_i \delta(t-\rho_j) \right)^{[n]} = - \left( \Omega D(\rho_i)\vec{\boldsymbol{r}}_i -  \tilde{B}(\rho_i)L(\rho_i)\boldsymbol{\xi}_j(\rho_i) \cev{\boldsymbol{r}}_j R'(\rho_i)\vec{\boldsymbol{r}}_i \right)^{[n]}.
    \end{multline*}
    It is easy to see now that, if we act with an eigenvector of $R(x)$ and take the limit $x \to \rho_i$ on \eqref{Eq:ConnectD}, the system is satisfied trivially and we cannot find the other $n$ unknowns. However, for $n=0$, the system can be solved by acting with the eigenvectors and eigenvalues of $L(x)$ on Equation \eqref{Eq: ConnectB}. One arrives at the following relation: 
   \[ \begin{aligned}
        \tilde{B}^{(b)}_0(x) & = \frac{1}{\tau_0} \sum_{\tilde{b}=1}^q\begin{vNiceMatrix}
            \mathbb{B}_0^{(1)} & \Cdots & \mathbb{B}_0^{(M^L)} & B_0^{(\tilde{b})}(x) \\
            \Vdots & & \Vdots & \Vdots \\\\\\
            \mathbb{B}_{M^L}^{(1)} & \Cdots & \mathbb{B}_{M^L}^{(M^L)} & B_{M^L}^{(\tilde{b})}(x)
        \end{vNiceMatrix} L^{-1}_{\tilde{b},b}(x), & \tau_0 & = \begin{vNiceMatrix}
            \mathbb{B}_0^{(1)} & \Cdots & \mathbb{B}_0^{(M^L)} \\
            \Vdots & & \Vdots \\\\
            \mathbb{B}_{M^L-1}^{(1)} & \Cdots & \mathbb{B}_{M^L-1}^{(M^L)}
        \end{vNiceMatrix}.
    \end{aligned}\]
\end{Remark}
\begin{Theorem} \label{Th: Existence}
   The existence of the Uvarov's perturbed orthogonality imply that $\tau_n \neq 0$ for $n \geq M^R$.
\end{Theorem}
\begin{proof}
    If the orthogonality exists and $\tau_n = 0$ for some $n \geq M^R$, then there exists a nonzero 
    vector $\boldsymbol{c}\in\C^{M^R+M^L}$ such that: 
    \begin{equation*}
        \begin{bNiceMatrix}
            \mathbb{B}_{n-M^R}^{(1)} & \Cdots & \mathbb{B}_{n-M^R}^{(M^L)} & \mathbb{D}_{n-M^R}^{(1)}-\mathbb{W}_{n-M^R}^{(1)} & \Cdots & \mathbb{D}_{n-M^R}^{(M^R)}-\mathbb{W}_{n-M^R}^{(M^R)} \\ \Vdots & & \Vdots & \Vdots & & \Vdots \\ \\\\
            \mathbb{B}_{n+M^L-1}^{(1)} & \Cdots & \mathbb{B}_{n+M^L-1}^{(M^L)} & \mathbb{D}_{n+M^L-1}^{(1)}-\mathbb{W}_{n+M^L-1}^{(1)} & \Cdots & \mathbb{D}_{n+M^L-1}^{(M^R)}-\mathbb{W}_{n+M^L-1}^{(M^R)}
        \end{bNiceMatrix} \boldsymbol{c} = 0. 
    \end{equation*}
    Let us write the vector as $\boldsymbol{c}=\begin{bNiceMatrix}
    \boldsymbol{c^L} & \boldsymbol{c^R}
    \end{bNiceMatrix}^\top$, where $\boldsymbol{c^L}$ and $\boldsymbol{c^R}$  have  $M^L$ and $M^R$ entries, respectively. From Proposition \ref{Prop: Kfinal}, we find: 
    \begin{equation*}
        R(z)\begin{bNiceMatrix}
            \mathbb{K}^{[n-1],(1)}(z) & \Cdots & \mathbb{K}^{[n-1],(M^L)}(z) & \mathbb{K}^{[n-1],(1)}_D(z) + \frac{\vec{\boldsymbol{r}}_1}{z-\rho_1} & \Cdots & \mathbb{K}^{[n-1],(M^R)}_D(z) + \frac{\vec{\boldsymbol{r}}_{M^R}}{z-\rho_{M^R}}
        \end{bNiceMatrix}\begin{bNiceMatrix}
            \boldsymbol{c^L} \\ \boldsymbol{c^R}
        \end{bNiceMatrix}=0.
    \end{equation*}
    For points where $z \neq \rho_i$, i.e. $\det R(z) \neq 0$, the expression is equivalent to: 
    \begin{equation} \label{Eq: DemNecesidad}
        \begin{bNiceMatrix}
            \mathbb{K}^{[n-1],(1)}(z) & \Cdots & \mathbb{K}^{[n-1],(M^L)}(z) & \mathbb{K}^{[n-1],(1)}_D(z) + \frac{\vec{\boldsymbol{r}}_1}{z-\rho_1} & \Cdots & \mathbb{K}^{[n-1],(M^R)}_D(z) + \frac{\vec{\boldsymbol{r}}_{M^R}}{z-\rho_{M^R}}
        \end{bNiceMatrix}\begin{bNiceMatrix}
            \boldsymbol{c^L} \\ \boldsymbol{c^R}
        \end{bNiceMatrix}=0.
    \end{equation}
    Let us integrate the last relation in the complex plane using a counterclockwise contour $C_i$ surrounding once $\rho_i$ and not any other eigenvalue. Since $\mathbb{K}^{[n-1],(j)}(z)$ and $\mathbb{K}^{[n-1],(k)}_D(z)$ are polynomial matrices for $j \in \{ 1, \dots , M^L\}$ and $k \in \{ 1, \dots , M^R\}$ and $(z-\rho_j)^{-1}$ is holomorphic inside the contour for $i \neq j$, we have: 
    \begin{multline*}
        \oint_{C_i} \d z \begin{bNiceMatrix}
            \mathbb{K}^{[n-1],(1)}(z) & \Cdots & \mathbb{K}^{[n-1],(M^L)}(z) & \mathbb{K}^{[n-1],(1)}_D(z) + \frac{\vec{\boldsymbol{r}}_1}{z-\rho_1} & \Cdots & \mathbb{K}^{[n-1],(M^R)}_D(z) + \frac{\vec{\boldsymbol{r}}_{M^R}}{z-\rho_{M^R}}
        \end{bNiceMatrix}\begin{bNiceMatrix}
            \boldsymbol{c^L} \\ \boldsymbol{c^R}
        \end{bNiceMatrix} \\ = 2\pi \mathrm{i} \begin{bNiceMatrix}
            0 & \Cdots & 0 & \vec{\boldsymbol{r}}_i & 0 & \Cdots & 0
        \end{bNiceMatrix}\begin{bNiceMatrix}
            \boldsymbol{c^L} \\ \boldsymbol{c^R}
        \end{bNiceMatrix} = \vec{\boldsymbol{r}}_i c^R_i = 0.
    \end{multline*}
    Which can only be true if $c^R_i = 0$. Taking into account all the eigenvalues of $R(x)$ we deduce  tha $\boldsymbol{c}^R = 0$. Simplifying Equation \eqref{Eq: DemNecesidad} yields: 
    \begin{equation*}
        \begin{bNiceMatrix}
            \mathbb{K}^{[n-1],(1)}(z) & \Cdots & \mathbb{K}^{[n-1],(M^L)}(z)
        \end{bNiceMatrix}\boldsymbol{c^L}=0.
    \end{equation*}
    By expanding the last relation, we find: 
    \begin{equation*}
        \sum_{i=0}^{n-1} \begin{bNiceMatrix}
            A_i^{(1)}(x) \\
            \Vdots \\
            A_i^{(p)}(x)
        \end{bNiceMatrix} \left( \mathbb{B}_i^{(1)}c_1^L + \dots + \mathbb{B}_i^{(M^L)}c_{M^L}^L \right) = 0. 
    \end{equation*}
    Due to the linear independence of the vectors $A_i(x)$, the expresion simplifies to: 
\[    \begin{aligned}
        \mathbb{B}_i^{(1)}c_1^L + \dots + \mathbb{B}_i^{(M^L)}c_{M^L}^L & = 0, & i &\in\{ 0, \dots, n-1\} . 
    \end{aligned}\]
    If $n\geq M^L$, we examine the relations for $i \in \{0, \dots M^L-1 \}$. Put together, they read: 
    \begin{equation*}
        \begin{bNiceMatrix}
            \mathbb{B}_0^{(1)} & \Cdots & \mathbb{B}_0^{(M^L)} \\
            \Vdots & & \Vdots \\\\\\
            \mathbb{B}_{M^L-1}^{(1)} & \Cdots & \mathbb{B}_{M^L-1}^{(M^L)}
        \end{bNiceMatrix} \boldsymbol{c^L} = 0, 
    \end{equation*}
    which is equivalent to say that $\tau_0 = 0$ (see remark \ref{Rm: n<M^R}) or that $\boldsymbol{c^L}=0$. The first condition implies that the vector $\tilde{B}_0$ is not defined, so the second must hold. Finally, we get that $\boldsymbol{c} = 0$, which is a contradiction.

    In the case where $M^R \leq n < M^L$, the prove must be slightly changed. On the one hand, we have that: 
 \[   \begin{aligned}
        \mathbb{B}_i^{(1)}c_1^L + \dots + \mathbb{B}_i^{(M^L)}c_{M^L}^L & = 0, & i& \in\{ 0, \dots, n-1\} . 
    \end{aligned}\]
    On the other hand, we find: 
    \begin{equation*}
        \begin{bNiceMatrix}
            \mathbb{B}_{n-M^R}^{(1)} & \Cdots & \mathbb{B}_{n-M^R}^{(M^L)} \\
            \Vdots & & \Vdots \\\\\\
            \mathbb{B}_{n+M^L-1}^{(1)} & \Cdots & \mathbb{B}_{n+M^L-1}^{(M^L)}
        \end{bNiceMatrix} \boldsymbol{c^L} = 0,
    \end{equation*}
    by the initial hypothesis and the fact that $\boldsymbol{c}^R = 0$. Taking into account these relations we can always construct once again the following relation: 
    \begin{equation*}
        \begin{bNiceMatrix}
            \mathbb{B}_0^{(1)} & \Cdots & \mathbb{B}_0^{(M^L)} \\
            \Vdots & & \Vdots \\\\\\\
            \mathbb{B}_{M^L-1}^{(1)} & \Cdots & \mathbb{B}_{M^L-1}^{(M^L)}
        \end{bNiceMatrix} \boldsymbol{c^L} = 0, 
    \end{equation*}
    and the result extends trivially to these cases. 
\end{proof}
        \subsection{Generalizations}\label{S:Uvarov multiple zeros}
In this section we seek to generalize our previous results. Firstly, we will drop the assumption on simple eigenvalues or simple roots of the perturbation matrices. Secondly, given the non-commutative character of the measure matrix, another perturbation can be studied, that is,
\begin{equation*}
    L(x) \d \tilde{\mu}(x) = \d \mu (x) R(x). 
\end{equation*}
Most of the results already proven extend naturally to both cases.
            \subsubsection{Eigenvalues of Arbitrary Multiplicity}
Let's consider the general perturbed measure given in Equation \eqref{Eq:GeneralMeasure}. In this scenario, it does not longer hold that the number of distinct roots of $R(x)$ and $L(x)$ are equal to $M^R$ and $M^L$, respectively. The connection matrix  $\Omega$ defined in Definition \ref{Prop: OmegaDefinition}, needs to be extended to the case by substituting $M^L$ and $M^R$ by $N^Lq-l$ and $N^Rp-r$, respectively. 

In what follows, the notation $\overleftarrow{\frac{\d^l}{\d x^l}}$ indicates that the $l$-th derivative acts on the functions of $x$ to its left; that is, it differentiates only those $x$-dependent factors appearing immediately before the operator:
\[
F(x) \overleftarrow{\dfrac{\d^l}{\d x^l}}G(x)={\dfrac{\d^lF}{\d x^l}}(x)G(x)
\]

At this point, let us introduce some further notation. 
\begin{Definition} \label{Def: NotationGeneralise1}
    We define the following expressions: 
    \begin{align*}
    \mathbb{D}_{n,j}^{(i)}& \coloneq \begin{bNiceMatrix}
        D_n(\rho_i)\vec{\boldsymbol{r}}_{i,j;0} & \Cdots & \displaystyle\sum_{l=0}^{\kappa^{R}_{i,j}-1} \frac{1}{l!} D^{(l)}_n(\rho_i)\vec{\boldsymbol{r}}_{i,j;\kappa^{R}_{i,j}-1-l}
    \end{bNiceMatrix}, \\
    \mathbb{W}_{n,j}^{(i)}& \coloneq \begin{multlined}[t][.75\textwidth]
    	\lim_{x \rightarrow \rho_i}
    \int_{-\infty}^\infty\d t	 \sum_{\tilde{\jmath}=1}^{s^R_i}\sum_{\tilde{k}=0}^{\kappa^{R}_{i,\tilde{\jmath}}-1}\sum_{\tilde{l}=0}^{\tilde{k}} \frac{(-1)^{\tilde{l}}}{\tilde{l}!} B(t) \boldsymbol{\xi}_{i,\tilde{\jmath},\tilde{k}}(t)  \cev{\boldsymbol{r}}_{i,\tilde{\jmath};\tilde{k}-\tilde{l}} \frac{\delta^{(\tilde{l})}(t-\rho_i)R(x)}{(x-t)}\\
    \times	\begin{bNiceMatrix}
        \vec{\boldsymbol{r}}_{i,j,0} & \Cdots &\displaystyle \sum_{l=0}^{\kappa^{R}_{i,j}-1} \frac{1}{l!} \overleftarrow{\dfrac{\d^l}{\d x^l}} \vec{\boldsymbol{r}}_{i,j,\kappa^{R}_{i,j}-1-l}
    \end{bNiceMatrix}, 
    \end{multlined} \\
    \mathbb{B}_{n,j}^{(k)} & \coloneq \begin{bNiceMatrix}
       B_n(\lambda_k)\vec{\boldsymbol{l}}_{k,j;0} & \Cdots & \displaystyle\sum_{l=0}^{\kappa^{L}_{k,j}-1} \frac{1}{l!} B^{(l)}_n(\lambda_k)\vec{\boldsymbol{l}}_{k,j;\kappa^{L}_{k,j}-1-l}
    \end{bNiceMatrix},
\end{align*}
for $n \in \mathbb{N}_0$, $i \in \{ 1, \dots M^R \}$ and $k \in \{ 1, \dots, M^L \}$. Here $\mathbb{D}_{n,j}^{(i)}$ and $\mathbb{W}_{n,j}^{(i)}$ are row vectors with  $\kappa_{i,j}^R$ entries and $\mathbb{B}_{n,j}^{(k)}$ are row vectors with $\kappa_{k,j}^L$ entries. Additionally, we define the vectors: 
\begin{align*}
    \mathbb{D}_n^{(i)}-\mathbb{W}_n^{(i)} & \coloneq \begin{bNiceMatrix}
        \mathbb{D}_{n,1}^{(i)}-\mathbb{W}_{n,1}^{(i)}& \Cdots & \mathbb{D}_{n,s^R_i}^{(i)}-\mathbb{W}_{n,s^R_i}^{(i)} \end{bNiceMatrix}, \\
    \mathbb{B}_n^{(k)} & \coloneq \begin{bNiceMatrix}
        \mathbb{B}_{n,1}^{(k)}& \Cdots & \mathbb{B}_{n,s^L_k}^{(k)} \end{bNiceMatrix},
\end{align*}
which have lengths $K^R_i$ and $K^L_i$, respectively.
\end{Definition}
\begin{Proposition} \label{Prop: LimitExistence}
    The following limit:
    \begin{equation*}
\begin{aligned}
	        &\lim_{x \rightarrow \rho_i} \check{D}(x) R(x) \left( \sum_{l=0}^k \frac{1}{l!}\overleftarrow{\frac{\d^l}{\d x^l}} \vec{\boldsymbol{r}}_{i,j;k-l}\right), & k&\in  \{ 0, \dots, \kappa^{R}_{i,j}-1 \},
\end{aligned}
    \end{equation*}
    exists and depends only on $R(x)$ and its derivatives, the canonical set of Jordan chains associated with the eigenvalue $\rho_i$ (both left and right), the matrix polynomial $\check{B}(x)$, and the vector functions $\boldsymbol{\xi}_{i,j,k}(x)$.
\end{Proposition}
\begin{proof}
    See Proposition 2.7 in \cite{Manas_Rojas_Geronimus}.
\end{proof}
\begin{Proposition}
    Proposition \ref{Prop: OmegaComp} holds in this scenario for $n \geq N^Rp-r$, simply by introducing the notation mentioned in Definition \ref{Def: NotationGeneralise1}.
\end{Proposition}
\begin{proof}
Proposition \ref{Prop: OmegaComp} extends to this case in the following way. For $n \geq N^Rp-r$ and using Equation \eqref{Eq: ConnectB}, we get,
\begin{equation*}
    \Omega B(x) \begin{bNiceMatrix}
        \vec{\boldsymbol{l}}_{k,j;0} & \Cdots & \displaystyle\sum_{l=0}^{\kappa^{L}_{k,j}-1} \frac{1}{l!} \overleftarrow{\dfrac{\d^l}{\d x^l}}\vec{\boldsymbol{l}}_{k,j;\kappa^{L}_{k,j}-1-l}
    \end{bNiceMatrix}_{x = \lambda_k} = \tilde{B}(x) L(x) \begin{bNiceMatrix}
        \vec{\boldsymbol{l}}_{k,j;0} & \Cdots & \displaystyle\sum_{l=0}^{\kappa^{L}_{k,j}-1} \frac{1}{l!} \overleftarrow{\dfrac{\d^l}{\d x^l}}\vec{\boldsymbol{l}}_{k,j;\kappa^{L}_{k,j}-1-l}
    \end{bNiceMatrix}_{x = \lambda_k} = 0,
\end{equation*}
and rearranging the equations for the different $s^R_i$, we have,
\begin{equation*}
    \left( \Omega \begin{bNiceMatrix}
        \mathbb{B}^{(k)}_{1} & \Cdots & \mathbb{B}^{(k)}_{s_k^L}
    \end{bNiceMatrix} \right)_n = \left( \Omega \mathbb{B}^{(k)} \right)_n =0. 
\end{equation*}
In Equation \eqref{Eq:ConnectD}, we proceed similarly. The integral term vanishes for $n \geq N^Rp-r$, we then find,  
\begin{equation*}
    \lim_{x \rightarrow \rho_i} \tilde{D}(x)R(x)\begin{bNiceMatrix}
        \vec{\boldsymbol{r}}_{i,j;0} & \Cdots & \displaystyle\sum_{l=0}^{\kappa^{R}_{i,j}-1} \frac{1}{l!} \overleftarrow{\dfrac{\d^l}{\d x^l}}\vec{\boldsymbol{r}}_{i,j;\kappa^{R}_{i,j}-1-l}
    \end{bNiceMatrix} = \Omega D(x)\begin{bNiceMatrix}
        \vec{\boldsymbol{r}}_{i,j;0} & \Cdots & \displaystyle\sum_{l=0}^{\kappa^{R}_{i,j}-1} \frac{1}{l!} \overleftarrow{\dfrac{\d^l}{\d x^l}}\vec{\boldsymbol{r}}_{i,j;\kappa^{R}_{i,j}-1-l}
    \end{bNiceMatrix}_{x = \rho_i}. 
\end{equation*}
The limit on the right hand side exists, see Proposition \ref{Prop: LimitExistence}, therefore, the last relation can be expressesed as: 
\begin{equation*}
    \left( \Omega \begin{bNiceMatrix}
        \mathbb{D}_{1}^{(i)}-\mathbb{W}_{1}^{(i)}& \Cdots & \mathbb{D}_{s^R_i}^{(i)}-\mathbb{W}_{s^R_i}^{(i)} \end{bNiceMatrix} \right)_n =  \left( \Omega \left[ \mathbb{D}^{(i)}-\mathbb{W}^{(i)} \right] \right)_n = 0. 
\end{equation*}
Taking into account the different roots of $\det L(x)$ and $\det R(x)$, we would find a linear system of $N^Lq-l+N^Rp-r$ equations for the same number of unknowns. 
\end{proof}
\begin{Definition} \label{Def: NotationGeneralise2}
    We also introduce the vectors: 
\[    \begin{aligned}
        \mathbb{K}^{[n],(k)}(x) & \coloneq \sum_{j=0}^{n}A_j(x) \mathbb{B}_j^{(k)}, &  \mathbb{K}^{[n],(i)}_D(x) & \coloneq \sum_{j=0}^{n}A_j(x) \left( \mathbb{D}_j^{(i)}-\mathbb{W}_j^{(i)} \right),
    \end{aligned}\]
    for $n \in \mathbb{N}_0$, $i \in \{ 1, \dots M^R \}$ and $k \in \{ 1, \dots, M^L \}$. Where $\mathbb{K}^{[n],(k)}(x)$ and  $\mathbb{K}_D^{[n],(i)}(x)$ are row vectors with $K^L_k$  and  $K^R_i$ entries, respectively. We will also make use of the following vectors: 
    \begin{align*}
    \mathbb{P}_j^{(i)}(x) & \coloneq \begin{bNiceMatrix}
        \dfrac{\vec{\boldsymbol{r}}_{i,j;0}}{x-\rho_i} & \Cdots & \displaystyle\sum_{l=0}^{\kappa^R_{i,j}-1} \frac{1}{l!} \frac{\d^l}{\d x^l} \left( \frac{1}{x-\rho_i} \right) \vec{\boldsymbol{r}}_{i,j;\kappa^R_{i,j}-1-l}
    \end{bNiceMatrix}, \\
    \mathbb{P}^{(i)}(x) & \coloneq \begin{bNiceMatrix}
        \mathbb{P}_1^{(1)}(x) & \Cdots & \mathbb{P}_{s^R_i}^{(i)}(x)
    \end{bNiceMatrix},
\end{align*}
for $i \in \{ 1, \dots, M^R \}$. In this case, $\mathbb{P}_j^{(i)}(x)$ and $\mathbb{P}^{(i)}(x)$ row vectors with $\kappa_{i,j}^R$ and $K_i^R$ entries, respectively.
\end{Definition}
Another crucial definition of our study are the $\tau$-determinants. Let us see how these determinants extend to this case.
\begin{Definition}
    The $\tau$-determinants can be expressed as:
    \begin{equation*}
    \tau_n \coloneqq \begin{vNiceMatrix}
            \mathbb{B}_{n-N^Rp+r}^{(1)} & \Cdots & \mathbb{B}_{n-N^Rp+r}^{(N^Lq-l)} & \mathbb{D}_{n-N^Rp+r}^{(1)}-\mathbb{W}_{n-N^Rp+r}^{(1)} & \Cdots & \mathbb{D}_{n-N^Rp+r}^{(N^Rp-r)}-\mathbb{W}_{n-N^Rp+r}^{(N^Rp-r)} \\ \Vdots & & \Vdots & \Vdots & & \Vdots \\ \\\\
            \mathbb{B}_{n+N^Lq-l-1}^{(1)} & \Cdots & \mathbb{B}_{n+N^Lq-l-1}^{(N^Lq-l)} & \mathbb{D}_{n+N^Lq-l-1}^{(1)}-\mathbb{W}_{n+N^Lq-l-1}^{(1)} & \Cdots & \mathbb{D}_{n+N^Lq-l-1}^{(N^Rp-r)}-\mathbb{W}_{n+N^Lq-l-1}^{(N^Rp-r)}
        \end{vNiceMatrix},
    \end{equation*}
    which holds for $n \geq N^Rp-r$.
\end{Definition}
\begin{Theorem}
    The Christoffel formulas hold for $n \geq N^Rp-r$. Equation \eqref{Eq: ChrisFormulasB} is extended by introducing the notation in \ref{Def: NotationGeneralise1} and \ref{Def: NotationGeneralise2} and Equation \eqref{Eq: ChrisFormulasA} takes the form: 
    \[\hspace{-1.5cm}
    \begin{multlined}[t][\textwidth]
        \tilde{A}_{n-1}^{(a)}(x) = \frac{1}{\tau_n}\sum_{\tilde{a}=1}^pR_{a,\tilde{a}}(x) \\ \times { \begin{vNiceMatrix}
        \mathbb{B}_{n-N^Rp+r}^{(1)} & \Cdots & \mathbb{B}_{n-N^Rp+r}^{(M^L)} & \mathbb{D}_{n-N^Rp+r}^{(1)}-\mathbb{W}_{n-N^Rp+r}^{(1)} & \Cdots & \mathbb{D}_{n-N^Rp+r}^{(M^R)}-\mathbb{W}_{n-N^Rp+r}^{(M^R)} \\ 
        \Vdots & & \Vdots & \Vdots & & \Vdots \\ \\\\
        \mathbb{B}_{n+N^Lq-l-2}^{(1)} & \Cdots & \mathbb{B}_{n+N^Lq-l-2}^{(M^L)} & \mathbb{D}_{n+N^Lq-l-2}^{(1)}-\mathbb{W}_{n+N^Lq-l-2}^{(1)} & \Cdots & \mathbb{D}_{n+N^Lq-l-2}^{(M^R)}-\mathbb{W}_{n+N^Lq-l-2}^{(M^R)}\\[4pt]
        \mathbb{K}^{[n-1-N^Rp+r],(1)}(x) & \Cdots & \mathbb{K}^{[n-1-N^Rp+r],(M^L)}(x) & \mathbb{K}^{[n-1-N^Rp+r],(1)}_D(x) + \mathbb{P}^{(1)} & \Cdots & \mathbb{K}^{[n-1-N^Rp+r],(M^R)}_D(x) + \mathbb{P}^{(M^R)}
    \end{vNiceMatrix}}. 
    \end{multlined}\]
\end{Theorem}
\begin{proof}
    Proposition \ref{Prop: KernelConex} holds, and the computations to arrive at the stated result is analogue to the ones already carried out. However, the term, 
    \begin{equation*}
        \frac{R(x) - R(y)}{x-y},
    \end{equation*}
    in Equation \eqref{Eq: MixedKernelConex}, needs extra consideration. We have,
    \begin{multline*}
        \frac{R(x) - R(y)}{x-y} \begin{bNiceMatrix}
        \vec{\boldsymbol{r}}_{i,j;0} & \Cdots & \displaystyle\sum_{l=0}^{\kappa^{R}_{i,j}-1} \frac{1}{l!} \overleftarrow{\dfrac{\d^l}{\d y^l}}\vec{\boldsymbol{r}}_{i,j;\kappa^{R}_{i,j}-1-l}
    \end{bNiceMatrix}_{y = \rho_i} \\ = R(x) \begin{bNiceMatrix}
        \dfrac{\vec{\boldsymbol{r}}_{i,j;0}}{x-\rho_i} & \Cdots & \displaystyle\sum_{l=0}^{\kappa^R_{i,j}-1} \frac{1}{l!} \frac{\d^l}{\d y^l} \left( \frac{1}{x-y} \right)_{y=\rho_i} \vec{\boldsymbol{r}}_{i,j;\kappa^R_{i,j}-1-l}
    \end{bNiceMatrix} = R(x) \mathbb{P}^{(i)}(x).
    \end{multline*}
\end{proof}
            \subsubsection{Dual Perturbation} \label{Dual Perturbation}
Let's turn our attention to the following dual perturbation, 
\begin{equation*}
    L(x) \d \tilde{\mu}(x) = \d \mu(x) R(x), 
\end{equation*}
where the degrees of $\det R(x)$ and $\det L(x)$ are $N^Rp-r$ and $N^Lq-l$ and they satisfy conditions \ref{Cond: CondicionesMatricesLideresFinales_1} and \ref{Cond: CondicionesMatricesLideresFinales_2}, respectively. The matrix polynomial $L(x)$ is subjected to the condition that $\sigma(L) \cap \Delta = \varnothing$. The perturbed linear functional can be decomposed as,
\begin{equation*}
    \tilde{\mu}(x) = L^{-1}(x) \mu(x) R(x) + \sum_{i=1}^{M^L}\sum_{j=1}^{s^L_i}\sum_{k=0}^{\kappa^L_{i,j}-1} \left( \sum_{l=0}^k\frac{(-1)^l}{l!} \vec{\boldsymbol{l}}_{i,j;k-l}\delta^{(l)}(x-\lambda_i)\right) \boldsymbol{\xi}_{i,j,k}(x)R(x),
\end{equation*}
where, similarly as before, $\delta$ denotes Dirac's distribution, $\vec{\boldsymbol{l}}_{i,j;k}$ are the corresponding column vectors associated to a canonical set of right Jordan chains of $L(x)$ and $\boldsymbol{\xi}_i$ are arbitrary row vector functions with $p$ entries. For our purpose, it is useful to define the matrix polynomials as:
\[\begin{aligned}
    A(x) & = X^{\top}_{[p]}(x) \Bar{S}^{\top}, & B(x) & = H^{-1}SX_{[q]}(x). 
\end{aligned}\]
This is just a normalization condition, and now $A(x)$ orthogonal polynomials are monic matrix polynomials.

In what follows, we will also assume simple roots for $\det R(x)$ and $\det L(x)$, i.e., $M^R = N^Rp-r$ and $M^L = N^Lq - l$.

\begin{Proposition}
    In this context, the connection matrix satisfies the following relation: 
    \begin{equation*}
        \Omega = H^{-1}SL(\Lambda_{[q]})\tilde{S}^{-1}\tilde{H} = \Bar{S}^{-\top} R(\Lambda_{[p]}^\top) \tilde{\Bar{S}}^\top.
    \end{equation*}
    The connection matrix is a banded matrix with at most $M^R$ nonzero sudbiagonals and $M^L$ nonzero superdiagonals. Furthermore, the last subdiagonal is populated by ones. Explicitly,
    \begin{equation*}
        \Omega = \begin{bNiceMatrix}
            \Omega_{0,0} & \Omega_{0,1} & \Cdots & \Omega_{0,M^L}  & 0 & \Cdots[shorten-end=5pt] \\
            \Omega_{1,0} & \Omega_{1,1} & \Cdots & \Omega_{1,M^L} & \Omega_{1,M^L+1} & \Ddots[shorten-end=10pt] \\
            \Vdots & \Vdots & & \Vdots & & \Ddots[shorten-end=-35pt]  & \Ddots[shorten-end=50pt] \\
            \Omega_{M^R-1,0} & \Omega_{M^R-1,1} & \Cdots & \Omega_{M^R-1,M^L-1} & \Cdots & & \Omega_{M^R-1,M^R+M^L-1} &   &\\
            1 & & & & & & & \Ddots[shorten-end=5pt] & \\
            0 & \Ddots[shorten-end=-40pt] \\
            \Vdots & \Ddots[shorten-end=5pt]
        \end{bNiceMatrix}.
    \end{equation*}
\end{Proposition}
\begin{Proposition}
    The connection formulas for the $\tilde{A}(x)$ and $\tilde{B}(x)$ are, 
\[    \begin{aligned}
        \Omega\tilde{B}(x) & = B(x)L(x), & A(x)\Omega & = R(x)\tilde{A}(x).  
    \end{aligned}\]
    For the Cauchy transforms, we find the following relations: 
    \begin{align*}
        L(z)\tilde{C}(z) & = C(z)\Omega + \int_\Delta \frac{L(z) - L(x)}{z-x} \d \tilde{\mu}(x) \tilde{A}(x), & D(z)R(z) & = \Omega \tilde{D}(z) + \int_\Delta B(x) \d \mu(x) \frac{R(z)-R(x)}{z-x}.
    \end{align*}
\end{Proposition}
In this case, we will make use of the relations between the $A(x)$ and the $C(x)$ families. For that reason, we generalize the notation introduced throughout this section. 
\begin{Definition} \label{Def: DualNotation}
    Let us introduce the following notation: 
  \[  \begin{aligned}
         \mathbb{A}_n^{(k)} & = \cev{\boldsymbol{r}}_kA_n(\rho_k), & \mathbb{C}^{(i)}_n & = \cev{\boldsymbol{l}}_iB_n(\lambda_i), & \mathbb{W}_n^{(i)} = \cev{\boldsymbol{l}}_iL'(\lambda_i)\vec{\boldsymbol{l}}_i\boldsymbol{\xi}_i(\lambda_i)A_n(\lambda_i).
    \end{aligned}\]
    For the CD kernels, we have:
\[    \begin{aligned}
        \mathbb{K}^{[n],(k)}(y) & = \cev{\boldsymbol{r}}_kK^{[n]}(\rho_k,y), & \mathbb{K}_C^{[n],(i)}(y) = \cev{\boldsymbol{l}}_iK_C^{[n]}(\lambda_i,y)-\cev{\boldsymbol{l}}_iL'(\lambda_i)\vec{\boldsymbol{l}}_i\boldsymbol{\xi}_i(\lambda_i)K^{[n]}(\lambda_i,y), 
    \end{aligned}\]
    where in both cases,  $n \in \mathbb{N}_0$, $k \in \{ 1, \dots M^R \}$ and $i \in \{ 1, \dots, M^L \}$.
\end{Definition}
\begin{Theorem} \label{Th: ChristFormulasDual}
    The Christoffel-type formulas for a transpose Geronimus perturbation are: 
    \begin{align*}
            \tilde{B}_{n-1}^{(b)}(y) &= \sum_{\tilde{b}=1}^q\begin{vNiceMatrix}
            \mathbb{A}_{n-M^L}^{(1)} & \Cdots & \mathbb{A}_{n+M^R-2}^{(1)} & \mathbb{K}_{\tilde{b}}^{[n-M^L-1],(1)}(y) \\
            \Vdots & & \Vdots & \Vdots \\\\\\
            \mathbb{A}_{n-M^L}^{(M^R)} & \Cdots & \mathbb{A}_{n+M^R-2}^{(M^R)} & \mathbb{K}_{\tilde{b}}^{[n-M^L-1],(M^R)}(y) \\
            \mathbb{C}_{n-M^L}^{(1)}-\mathbb{W}_{n-M^L}^{(1)} & \Cdots & \mathbb{C}_{n+M^R-2}^{(1)}-\mathbb{W}_{n+M^R-2}^{(1)} & \mathbb{K}_{C,\tilde{b}}^{[n-M^L-1],(1)}(y) + \frac{\cev{l}_{1,\tilde{b}}}{y-\lambda_1} \\
            \Vdots & & \Vdots & \Vdots \\\\\\
            \mathbb{C}_{n-M^L}^{(M^L)}-\mathbb{W}_{n-M^L}^{(M^L)} & \Cdots & \mathbb{C}_{n+M^R-2}^{(M^L)}-\mathbb{W}_{n+M^R-2}^{(M^L)} & \mathbb{K}_{C,\tilde{b}}^{[n-M^L-1],(M^L)}(y) + \frac{\cev{l}_{M^L,\tilde{b}}}{y-\lambda_{M^L}}
        \end{vNiceMatrix}\frac{L_{\tilde b,b}(y)}{\tau_n}, \\[4pt]
        \tilde{A}_n^{(a)}(x) & = \frac{1}{\tau_n} \sum_{\tilde{a}=1}^p R^{-1}_{\tilde{a},a}(x) \begin{vNiceMatrix}
            \mathbb{A}_{n-M^L}^{(1)} & \Cdots & \mathbb{A}_{n+M^R}^{(1)} \\
            \Vdots & & \Vdots \\\\
            \mathbb{A}_{n-M^L}^{(M^R)} & \Cdots & \mathbb{A}_{n+M^R}^{(M^R)} \\
            \mathbb{C}_{n-M^L}^{(1)}-\mathbb{W}_{n-M^L}^{(1)} & \Cdots & \mathbb{C}_{n+M^R}^{(1)}-\mathbb{W}_{n+M^R}^{(1)} \\
            \Vdots & & \Vdots \\\\
            \mathbb{C}_{n-M^L}^{(M^L)}-\mathbb{W}_{n-M^L}^{(M^L)} & \Cdots & \mathbb{C}_{n+M^R}^{(M^L)}-\mathbb{W}_{n+M^R}^{(M^L)} \\[3pt]
            A_{n-M^L}^{(a)}(x) & \Cdots & A_{n+M^R}^{(a)}(x)
        \end{vNiceMatrix},
    \end{align*}
    where 
    \begin{equation*}
        \tau_n = \begin{vNiceMatrix}
            \mathbb{A}_{n-M^L}^{(1)} & \Cdots & \mathbb{A}_{n+M^R-1}^{(1)} \\
            \Vdots & & \Vdots \\\\\\
            \mathbb{A}_{n-M^L}^{(M^R)} & \Cdots & \mathbb{A}_{n+M^R-1}^{(M^R)} \\
            \mathbb{C}_{n-M^L}^{(1)}-\mathbb{W}_{n-M^L}^{(1)} & \Cdots & \mathbb{C}_{n+M^R-1}^{(1)}-\mathbb{W}_{n+M^R-1}^{(1)} \\
            \Vdots & & \Vdots \\\\\\
            \mathbb{C}_{n-M^L}^{(M^L)}-\mathbb{W}_{n-M^L}^{(M^L)} & \Cdots & \mathbb{C}_{n+M^R-1}^{(M^L)}-\mathbb{W}_{n+M^R-1}^{(M^L)}
        \end{vNiceMatrix}.
    \end{equation*}
\end{Theorem}
\subsection{Markov--Stieltjes Functions}\label{S:Markov-Stieltjes}

The Markov--Stieltjes function associated with a measure \(\d\mu\) corresponding to standard orthogonal polynomials is defined as  
\[
F(z) = \int_{\Delta}\frac{\d\mu(x)}{z-x},
\]  
which represents the Cauchy transform of the identity element and plays a fundamental role in the spectral theory of orthogonal polynomials. As demonstrated in \cite{Zhe}, linear spectral transformations of the Markov--Stieltjes function for standard orthogonal polynomials take the form  
\begin{equation} \label{Eq: MarkovStieljes}
    F(z) \longrightarrow \frac{A(z) F(z) + B(z)}{C(z)F(z) + D(z)},
\end{equation}
where \(A, B, C,\) and \(D\) are polynomials.  

Zhedanov established in \cite{Zhe} that Christoffel transformations induce a more general linear spectral transformation of the form  
\[
F(z) \longrightarrow A(z) F(z) + B(z).
\]  
Furthermore, he demonstrated that the composition of \(N\) successive Geronimus transformations yields a linear spectral transformation of the form  
\[
F(z) \longrightarrow \frac{F(z) + B(z)}{D(z)},
\]  
where \(\deg B = N-1\) and \(\deg D = N\).  

In this work, we extend these results to Uvarov transformations in the mixed-type multiple orthogonality scenario. Specifically, we show that Equation \eqref{Eq: MarkovStieljes} remains valid, but with \(A, B, C,\) and \(D\) now being matrix polynomials.  

Consider a Markov--Stieltjes \(q \times p\)-matrix function defined by  
\[
F(z) \coloneq \int_{\Delta}\frac{\d\mu(x)}{z-x}.
\]  

\begin{Proposition}
	The Uvarov transformation, \(\d \tilde{\mu}(x) R(x) = L(x) \d \mu(x)\) (where \(\deg L = N^L\) and \(\deg R = N^R\)), induces a matrix linear spectral transformation of the form  
	\[
	\tilde{F}(z) = \big(L(z)F(z) - S(z) + \tilde{S}(z) \big) R^{-1}(z),
	\]  
	with \(\deg S = N^L - 1\) and \(\deg \tilde{S} = N^R - 1\). The matrix polynomial coefficients are given by  
	\[
\begin{aligned}
		S_k&=\int_{\Delta} (L_{k+1}+xL_{k+2}+\cdots+x^{N_L-k-1}L_{N_L}) \d \mu(x), & k&\in\{0,\dots,N^L - 1\}, \\
        \tilde{S}_k&=\int_{\Delta} \d \tilde{\mu}(x) (R_{k+1}+xR_{k+2}+\cdots+x^{N_R-k-1}R_{N_R}), & k&\in\{0,\dots,N^R - 1\}.
\end{aligned}
		\]
\end{Proposition}

\begin{proof}
	The result follows directly from the relation \(\d \tilde{\mu}(x) R(x) = L(x) \d \mu(x)\). Specifically, we derive the identity  
	\[
	\tilde{F}(z) R(z) - \int_\Delta \d \tilde{\mu}(x) \frac{R(z) - R(x)}{z-x}  = L(z)F(z) - \int_\Delta \frac{L(z) - L(x)}{z-x} \d \mu(x).
	\]  
	By defining  
\[	\begin{aligned}
	    S(z) & \coloneq \int_\Delta \d \tilde{\mu}(x) \frac{R(z) - R(x)}{z-x}, & \tilde{S}(z) & \coloneq \int_\Delta \frac{L(z) - L(x)}{z-x} \d \mu(x), 
	\end{aligned}\]
	and applying Lemma \ref{Lema: R-BivaMP}, we obtain  
	\[
	\begin{aligned}
		S(z) &= \sum_{i=1}^{N^R} \sum_{j=1}^i \left( \int_{\Delta} \d\tilde{\mu}(x) x^{i-j} \right) R_i z^{j-1}, & 
        \tilde{S}(z) & = \sum_{i=1}^{N^L} L_i \sum_{j=1}^i \left( \int_{\Delta} \d\mu(x) x^{i-j} \right) z^{j-1},
	\end{aligned}
	\]  
	confirming that these are matrix polynomials of degrees \(N^R-1\) and \(N^L - 1\), respectively.
\end{proof}

\subsection{Case Study: Non-Trivial Perturbations of  Jacobi--Piñeiro}\label{S:case study}
As the explicit formulas for multiple orthogonal polynomials of mixed type are rare, we will use as case study a non mixed case.
Hence, we consider Jacobi--Piñeiro multiple orthogonal polynomials, that is $q=1$, and in particular, $p=3$.

 Based on the results shown in \cite{ExamplesMOP}, see also \cite{ExamplesMOP2}, we have $\d\mu_a=w_a\d\mu$, where the weight functions $w_a$ and measure $\d\mu$ are
\begin{align*}
&\begin{aligned}
 w_{a}(x;\alpha_a)=&x^{\alpha_a},& a&\in\{1,2,3\}, 
\end{aligned}&\d\mu(x)&=(1-x)^\beta\d x, & \Delta&=[0,1],
\end{align*}
with $\alpha_1,\alpha_2,\alpha_3,\beta>-1$ and, in order to have an AT system, $\alpha_i-\alpha_j\not\in\mathbb Z$ for $i\neq j$.

The moments are 
\[\begin{aligned}
\label{MomentJP}
\int_{0}^{1}x^{\alpha_a+k}(1-x)^\beta\d x&=\dfrac{\Gamma(\beta+1)\Gamma(\alpha_a+k+1)}{\Gamma(\alpha_a+\beta+k+2)},
& k &\in \mathbb N_0.
\end{aligned}\]

    The Pochhammer symbols $(x)_n$, $x\in\C$ and $n\in\N_0$, are defined as,
\begin{align*}
 (x)_n\coloneq\dfrac{\Gamma(x+n)}{\Gamma(x)}=\begin{cases}
 x(x+1)\cdots(x+n-1)\;\text{if}\;n\in\N,\\
 1\;\text{if}\;n=0.
 \end{cases}
\end{align*}

The step-line in this situation means that we are considering triples of nonnegative integers, say $(n_1, n_2, n_3)$, of the form
\begin{align*}
	(0, 0, 0), (1, 0, 0), (1, 1, 0), (1, 1, 1), (2, 1, 1), (2, 2, 1), \dots.
\end{align*}
On the step-line, we have, for $m \in \mathbb{N}_0$,
\[
\begin{aligned}
	n &= 3m, & (n_1, n_2, n_3) &= (m, m, m), \\
	n &= 3m + 1, & (n_1, n_2, n_3) &= (m + 1, m, m), \\
	n &= 3m + 2, & (n_1, n_2, n_3) &= (m + 1, m + 1, m).
\end{aligned}
\]

    The Jacobi--Piñeiro polynomials of type I are, see  \cite{ExamplesMOP},
    \begin{align*}
\label{JPTypeI}
 A_n^{(a)}(x) \coloneq P^{(a)}_{n}(x;\alpha_1,\alpha_2,\alpha_3,\beta) =\sum_{l=0}^{n_a-1}C_{n}^{(a),l}x^l,
\end{align*}
with
\begin{multline*}
C_{n}^{(a),l}\coloneq\frac{
(-1)^{n-1}
\prod_{q=1}^{3}(\alpha_q+\beta+n)_{n_q}}{
(n_a-1)! 
\prod_{q=1,q\neq a}^{3}(\alpha_q-\alpha_a)_{n_q}} 
\frac{\Gamma(\alpha_a+\beta+n)}{\Gamma(\beta+n)\Gamma(\alpha_a+1)}
\\ \times
\frac{(-n_a+1)_l(\alpha_a+\beta+n)_l}{l!(\alpha_a+1)_l}
\prod_{q=1,q\neq a}^3
\frac{(\alpha_a-\alpha_q-n_q+1)_l}{(\alpha_a-\alpha_q+1)_l},
\end{multline*}
and 
\[\begin{aligned}
    n_a & = \deg(A_n^{(a)})+1, & n & = n_1+n_2+n_3.
\end{aligned}\]

    The monic Jacobi--Piñeiro polynomials of type II are, see  \cite{ExamplesMOP},
    \begin{align}
        B_n(x) \coloneq P_{n}(x;\alpha_1,\alpha_2,\alpha_3,\beta)=\sum_{l_1=0}^{n_1}\sum_{l_2=0}^{n_2}\sum_{l_3=0}^{n_3}
        C_{n}^{l_1,l_2,l_3}\, x^{l_1+l_2+l_3}
    \end{align}
    with
    \begin{multline*}
        C_{n}^{l_1,l_2,l_3}
        \coloneq
        (-1)^{n}\prod_{q=1}^3\dfrac{(\alpha_q+1)_{n_q}}{(\alpha_q+\beta+n+1)_{n_q}}\dfrac{(-n_q)_{l_q}}{l_q!}\dfrac{(\alpha_1+\beta+n_1+1)_{l_1+l_2+l_3}}{(\alpha_1+1)_{l_1+l_2+l_3}}
        \\
        \times\dfrac{(\alpha_1+n_1+1)_{l_2+l_3}(\alpha_{2}+n_{2}+1)_{l_3}}{(\alpha_1+\beta+n_1+1)_{l_2+l_3}(\alpha_{2}+\beta+n_1+n_{2}+1)_{l_3}}
        \\ \times \dfrac{(\alpha_2+\beta+n_1+n_2+1)_{l_2+l_3}(\alpha_{3}+\beta+n+1)_{l_3}}{(\alpha_2+1)_{l_2+l_3}(\alpha_{3}+1)_{l_3}}.
    \end{multline*}

Let us consider two perturbed measures defined as follows:
\[\begin{aligned}
    (c-x) \d \tilde{\mu}_1(x) & = \d \mu(x) P_1(x), & P_1(x) &= \begin{bNiceMatrix}
        0 & x-d & 0 \\
        0 & 0 & x-d \\
        1 & 0 & 0
    \end{bNiceMatrix}, \\  
    \d \tilde{\mu}_2(x) P_2(x) & = (x-d) \d \mu(x), & P_2(x) &= \begin{bNiceMatrix}
        0 & c-x & 0 \\
        0 & 0 & c-x \\
        1 & 0 & 0
    \end{bNiceMatrix}, 
\end{aligned}\]
where \( c, d \in (-\infty, 0] \cup [1, \infty) \). Note that Conditions \ref{Cond: CondicionesMatricesLideresFinales_1} and \ref{Cond: CondicionesMatricesLideresFinales_2} have been modified, but as previously mentioned, these were merely normalization conditions. The right eigenvectors of \( P_1(x) \) and \( P_2(x) \) can be chosen as \( \boldsymbol{e}_2 \) and \( \boldsymbol{e}_3 \) for both matrices, with corresponding left eigenvectors \( \boldsymbol{e}_1^\top \) and \( \boldsymbol{e}_2^\top \).

The perturbed measures can be expressed in terms of the original measure as:
\begin{align*}
    \d \tilde{\mu}_1(x) & = \d \mu(x) \begin{bNiceMatrix}
        0 & \dfrac{x-d}{c-x} & 0 \\
        0 & 0 & \dfrac{x-d}{c-x} \\
        \dfrac{1}{c-x} & 0 & 0
    \end{bNiceMatrix} + \delta(x-c) \boldsymbol{\xi}(x) \begin{bNiceMatrix}
        0 & x-d & 0 \\
        0 & 0 & x-d \\
        1 & 0 & 0
    \end{bNiceMatrix}, \\[3pt]
    \d \tilde{\mu}_2(x) & = \d \mu(x) \begin{bNiceMatrix}
        0 & 0 & x-d \\
        \dfrac{x-d}{c-x} & 0 & 0 \\
        0 & \dfrac{x-d}{c-x} & 0
    \end{bNiceMatrix} + \delta(x-c) (x-d)\left( \xi^{(1)}(x) \boldsymbol{e}_1^\top + \xi^{(2)}(x) \boldsymbol{e}_2^\top \right), 
\end{align*}
where \( \boldsymbol{\xi}(x) \) is an arbitrary \( 1 \times 3 \) matrix function and \( \xi^{(1)}(x), \xi^{(2)}(x) \) are arbitrary functions of \( x \). Setting \( c = 1 \) and \( d = 0 \) for the first perturbation (with \( \boldsymbol{\xi}(x) = 0 \)) recovers a known family of Jacobi–Piñeiro polynomials. In this case, the parameters transform as
\[
(\alpha_1, \alpha_2, \alpha_3, \beta)  \xrightarrow{}{}  (\alpha_3, \alpha_1+1, \alpha_2+1, \beta - 1),
\]
provided that \( \beta > 0 \). However, this parameter choice for the second measure does not yield Jacobi–Piñeiro orthogonal polynomials, as the parameter \( \beta \) affects the entries of \( \d \tilde{\mu}_2(x) \) differently. Note that this choice of \( c \) and \( d \) places the zeros of the perturbation precisely at the boundary of the support, that is, within the interval \( [0,1] \). However, local integrability is guaranteed by the specific choice of parameters.

\begin{Proposition}
    For \( n > 0 \), the Christoffel-type formula for the type I multiple orthogonal polynomials associated with \( \d \tilde{\mu}_1 \) is given by:
    \[
    \left( \begin{bNiceMatrix}
        (x-d)\tilde{P}_{n-1}^{(2)}(x) \\[2pt]
        (x-d)\tilde{P}_{n-1}^{(3)}(x) \\[2pt]
        \tilde{P}_{n-1}^{(1)}(x) 
    \end{bNiceMatrix} \right)_{a} = \frac{1}{\tau_n} \begin{vNiceMatrix}
        P_{n-1}^{(a)}(x) & P_{n-1}^{(1)}(d) & P_{n-1}^{(2)}(d) & W_{n-1}(c) \\[3pt]
        P_{n}^{(a)}(x) & P_{n}^{(1)}(d) & P_{n}^{(2)}(d) & W_{n}(c) \\[3pt]
        P_{n+1}^{(a)}(x) & P_{n+1}^{(1)}(d)  & P_{n+1}^{(2)}(d) & W_{n+1}(c) \\[3pt]
        P_{n+2}^{(a)}(x) & P_{n+2}^{(1)}(d) & P_{n+2}^{(2)}(d) & W_{n+2}(c)
    \end{vNiceMatrix}.
    \] 
    For type II polynomials, we have:
    \[
    \tilde{P}_{n+1}(x) = (x-c)B_n(y) + \frac{(x-c)}{\tau_n} \sum_{i=0}^{n-1} \begin{vNiceMatrix}
        P_{i}^{(1)}(d) & P_{n+1}^{(1)}(d) & P_{n+2}^{(1)}(d) \\[3pt]
        P_{i}^{(2)}(d) & P_{n+1}^{(2)}(d) & P_{n+2}^{(2)}(d) \\[3pt]
        W_i(c) & W_{n+1}(c) & W_{n+2}(c) 
    \end{vNiceMatrix} B_i(x) - \frac{1}{\tau_n} \begin{bNiceMatrix}
        P_{n+1}^{(1)}(d) & P_{n+2}^{(1)}(d) \\[3pt]
        P_{n+1}^{(2)}(d) & P_{n+2}^{(2)}(d) 
    \end{bNiceMatrix},
    \]
    where, in both cases,
    \[
    W_i(c) \coloneq C_i(c) + \sum_{j=1}^3 \xi_j(c) P_{i}^{(j)}(c), \quad \tau_n \coloneq \begin{vNiceMatrix}
        P_{n}^{(1)}(d) & P_{n+1}^{(1)}(d) & P_{n+2}^{(1)}(d) \\[3pt]
        P_{n}^{(2)}(d) & P_{n+1}^{(2)}(d) & P_{n+2}^{(2)}(d) \\[3pt]
        W_{n}(c) & W_{n+1}(c) & W_{n+2}(c) 
    \end{vNiceMatrix}.
    \]
\end{Proposition}

\begin{proof}
 The measure   $\d \tilde{\mu}_1$ corresponds to the dual perturbation described in \S \ref{Dual Perturbation}, where the terms $B(x)$ are kept monic. Moreover, the left perturbation polynomial has leading coefficient $-1$. These considerations lead to modifications in the calculations. For this perturbation, the connection matrix takes the form: 
\[
		\Omega = \begin{bNiceMatrix}
			\Omega_{0,0} & -1 & 0 & \Cdots[shorten-end=3pt] \\
			\Omega_{1,0} & \Omega_{1,1} & -1 & 0 & \Cdots[shorten-end=3pt] \\
			\Omega_{2,0} & \Omega_{2,1} & \Omega_{2,2} & -1 & 0 & \Cdots[shorten-end=3pt] \\
			0 & \Ddots[shorten-end=-25pt]  & & & \Ddots[shorten-end=-25pt]  & \Ddots[shorten-end=-5pt]  \\
			\Vdots[shorten-end=-2pt] & \Ddots[shorten-end=-5pt] & \phantom{i} & & & \phantom{i} & \phantom{i} 
		\end{bNiceMatrix},
\]
which yields the following connection formulas:
\begin{align*}
A(x)\Omega & = P_1(x)\tilde{A}(x), & (c-x)\tilde{C}(x) & = C(x)\Omega - \int_\Delta \d \tilde{\mu}_1(x) \tilde{A}(t).
\end{align*}

For the CD kernels and the mixed-type CD kernels, additional modifications also arise, in particular: 
\begin{multline*}
	P_1(x)\tilde{K}^{[n]}(x,y) = K^{[n]}(x,y)(c-y) + A(x) \left[ \Omega, \Pi_n \right]\tilde{B(y)}, \\
	(c-x)\tilde{K}^{[n]}_C(x,y) = K_C^{[n]}(x,y)(c-y) + C(x)\left[ \Omega, \Pi_n \right] -1 . 
\end{multline*} 

Taking these modifications into account, the subsequent calculations proceed analogously to those previously discussed.

\end{proof}

\begin{Proposition}
    For \( n > 1 \), the Christoffel-type formula for the type I multiple orthogonal polynomials associated with \( \d \tilde{\mu}_2 \) is:
    \begin{multline*}
    \begin{bNiceMatrix}
        \tilde{P}_{n-1}^{(1)}(x) \\[2pt]
        \tilde{P}_{n-1}^{(2)}(x) \\[2pt]
        \tilde{P}_{n-1}^{(3)}(x) 
    \end{bNiceMatrix} = \frac{1}{\tau_n} \sum_{i=0}^{n-3} \begin{bNiceMatrix}
        (x-c) P_{i}^{(2)}(x) \\[2pt]
        (x-c) P_{i}^{(3)}(x) \\[2pt]
        -P_{i}^{(1)}(x)
    \end{bNiceMatrix} \begin{vNiceMatrix}
        P_{i}(d) & D_{i}^{(2)}(c)+\xi^{(1)}(c) P_{i}(c) & D_{i}^{(3)}(c)+\xi^{(2)}(c) P_{i}(c) \\[3pt]
        P_{n-2}(d) & D_{n-2}^{(2)}(c)+\xi^{(1)}(c) P_{n-2}(c) & D_{n-2}^{(3)}(c)+\xi^{(2)}(c) P_{n-2}(c) \\[3pt]
        P_{n-1}(d) & D_{n-1}^{(2)}(c)+\xi^{(1)}(c) P_{n-1}(c) & D_{n-1}^{(3)}(c)+\xi^{(2)}(c) P_{n-1}(c)
    \end{vNiceMatrix} \\
    + \frac{1}{\tau_n} \left( \begin{vNiceMatrix}
        P_{n-2}(d) & D_{n-2}^{(3)}(c)+\xi^{(2)}(c) P_{n-2}(c) \\[3pt]
        P_{n-1}(d) & D_{n-1}^{(3)}(c)+\xi^{(2)}(c) P_{n-1}(c)
    \end{vNiceMatrix} + \begin{vNiceMatrix}
        P_{n-2}(d) & D_{n-2}^{(2)}(c)+\xi^{(1)}(c) P_{n-2}(c) \\[3pt]
        P_{n-1}(d) & D_{n-1}^{(2)}(c)+\xi^{(1)}(c) P_{n-1}(c) 
    \end{vNiceMatrix} \right).
    \end{multline*}
    For type II polynomials, we have:
    \[
    \tilde{P}_{n}(x) = \frac{1}{\tau_n (x-d)} \begin{vNiceMatrix}
        P_{n-2}(d) & D_{n-2}^{(2)}(c)+\xi^{(1)}(c) P_{n-2}(c) & D_{n-2}^{(3)}(c)+\xi^{(2)}(c) P_{n-2}(c) & P_{n-2}(x) \\[3pt]
        P_{n-1}(d) & D_{n-1}^{(2)}(c)+\xi^{(1)}(c) P_{n-1}(c) & D_{n-1}^{(3)}(c)+\xi^{(2)}(c) P_{n-1}(c) & P_{n-1}(x) \\[3pt]
        P_{n}(d) & D_{n}^{(2)}(c)+\xi^{(1)}(c) P_{n}(c) & D_{n}^{(3)}(c)+\xi^{(2)}(c) P_{n}(c) & P_{n}(x) \\[3pt]
        P_{n+1}(d) & D_{n+1}^{(2)}(c)+\xi^{(1)}(c) P_{n+1}(c) & D_{n+1}^{(3)}(c)+\xi^{(2)}(c) P_{n+1}(c) & P_{n+1}(x)
    \end{vNiceMatrix},
    \]
    where, in both cases,
    \[
    \tau_n = \begin{vNiceMatrix}
        P_{n-2}(d) & D_{n-2}^{(2)}(c)+\xi^{(1)}(c) P_{n-2}(c) & D_{n-2}^{(3)}(c)+\xi^{(2)}(c) P_{n-2}(c) \\[3pt]
        P_{n-1}(d) & D_{n-1}^{(2)}(c)+\xi^{(1)}(c) P_{n-1}(c) & D_{n-1}^{(3)}(c)+\xi^{(2)}(c) P_{n-1}(c) \\[3pt]
        P_{n}(d) & D_{n}^{(2)}(c)+\xi^{(1)}(c) P_{n}(c) & D_{n}^{(3)}(c)+\xi^{(2)}(c) P_{n}(c)
    \end{vNiceMatrix}.
    \]
\end{Proposition}

\begin{proof}
    This result follows from Theorem \ref{Th: ChrisFormulas}, using Definitions \ref{Def:NotationBD} and \ref{Def: NotationK}. Two key observations are:
    \begin{enumerate}
        \item The terms \( \boldsymbol{e}_1^\top P_2'(c) \boldsymbol{e}_2 \) and \( \boldsymbol{e}_2^\top P_2'(c) \boldsymbol{e}_3 \) reduce to \(-1\), explaining the plus signs in \( D_{n}^{(2)}(c)+\xi^{(1)}(c) P_{n}(c) \) and \( D_{n}^{(3)}(c)+\xi^{(2)}(c) P_{n}(c) \).
        \item The normalization of the leading coefficient of \( P_2(x) \) introduces a global minus sign in the formula for type I polynomials.
    \end{enumerate}
\end{proof}

For the case \( c = 1 \) and \( d = 0 \), explicit formulas can be derived for the Jacobi–Piñeiro  multiple orthogonal polynomials of mixed type and their Cauchy transforms. The type II Jacobi–Piñeiro polynomials evaluated at the endpoints are:
\[
\begin{aligned}
    P_n(0) & = (-1)^n \prod_{a=1}^3 \frac{(\alpha_a+1)_{n_a}}{(\alpha_a+\beta+1)_{n_a}}, & 
    P_n(1) & = \frac{(\beta+1)_{n}}{\prod_{a=1}^3 (\alpha_a+\beta+n+1)_{n_a}},
\end{aligned}
\]
and for type I polynomials, we have:
\[
\begin{aligned}
    P_n^{(a)}(0) & = C_n^{(a),0} = \frac{(-1)^{n-1} \prod_{q=1}^{3} (\alpha_q+\beta+n)_{n_q}}{(n_a-1)! \prod_{q=1,q\neq a}^{3} (\alpha_q-\alpha_a)_{n_q}} \frac{\Gamma(\alpha_a+\beta+n)}{\Gamma(\beta+n)\Gamma(\alpha_a+1)}, \\ 
    P_n^{(a)}(1) & = \sum_{l=0}^{n_a-1} C_{n}^{(a),l}. 
\end{aligned}
\]

\begin{Proposition}
    The Cauchy transform for the Jacobi–Piñeiro polynomials can be explicitly computed at \( x = 1\). In particular:
    \[
    C_n(1) = \Gamma(\beta) \sum_{a=1}^3 \sum_{l=0}^{n_a-1} \frac{C_n^{(a),l}}{(\alpha_a+l+1)_\beta}, \quad 
    D_n^{(a)}(1) = \Gamma(\beta) \sum_{l_1=0}^{n_1} \sum_{l_2=0}^{n_2} \sum_{l_3=0}^{n_3} \frac{C_{n}^{l_1,l_2,l_3}}{(\alpha_a+l_1+l_2+l_3)_{\beta}}.
    \]
\end{Proposition}

\begin{proof}
    Consider the integral
    \[
    \int_0^1 x^{l_1+l_2+l_3+\alpha_2} (1-x)^{\beta-1} \d x = \frac{\Gamma(\beta)\Gamma(\alpha_2+l_1+l_2+l_3)}{\Gamma(\alpha_2+\beta+l_1+l_2+l_3+1)} = \frac{\Gamma(\beta)}{(\alpha_2+l_1+l_2+l_3)_{\beta}},
    \]
    which is valid for \( \beta > 0 \). The Cauchy transform of the type II Jacobi–Piñeiro polynomials can be expressed as:
    \[
    D_n^{(a)}(1; \alpha_1, \alpha_2, \alpha_3, \beta) = \int_0^1 P_n(x; \alpha_1, \alpha_2, \alpha_3, \beta) \frac{\d \mu_a(x)}{1-x} = \sum_{l_1=0}^{n_1} \sum_{l_2=0}^{n_2} \sum_{l_3=0}^{n_3} C_{n}^{l_1,l_2,l_3} \int_\Delta x^{l_1+l_2+l_3} \frac{\d \mu_a(x)}{1-x}.
    \]
    Substituting the integral result yields the desired relation. The first relation for \( C_n(1) \) follows similarly.
\end{proof}
    \section{On the Existence of Perturbed Orthogonality}\label{S:existence}
In Theorem \ref{Th: Existence}, we establish that if the orthogonality condition holds, then $\tau_n \neq 0$ for $n \geq M^R$\footnote{Throughout this section, we assume that the eigenvalues of $R(x)$ and $L(x)$ are simple. However, the general result can be analogously derived without this assumption.}. Although the converse implication has been proven in the settings of Christoffel and Geronimus perturbations, namely, $\mathrm{d}\hat{\mu}(x) = \mathrm{d}\mu(x) R(x)$ and $\mathrm{d}\check{\mu}(x) R(x) = \mathrm{d}\mu(x)$, respectively, it remains unresolved for the general Uvarov perturbation. A partial result, however, can be obtained under an additional assumption on the structure of the matrix $\Omega$. We now present the following result:

\begin{Theorem} \label{Th: Existence II}
    The perturbed orthogonality condition holds if the following criteria are satisfied:
    \begin{itemize}
        \item The connection matrix $\Omega$ admits an $LU$ factorization.
        \item $\tau_n \neq 0$ for all $n \geq M^R$.
        \item There exist \(M^R\) vector polynomials, \(\check{A}_n(x)\), of strictly increasing degree,
	\[
	\begin{aligned}
		\tilde{A}_0(x) & = \begin{bNiceMatrix}
		    \alpha \\ 0 \\ \Vdots \\ 0
		\end{bNiceMatrix}, & \tilde{A}_1(x) & = \begin{bNiceMatrix}
		    * \\ \alpha \\ 0 \\ \Vdots \\ 0
		\end{bNiceMatrix}, & \dots & & \tilde{A}_p(x) & = \begin{bNiceMatrix}
		    * \\ \Vdots \\ * \\ \alpha
		\end{bNiceMatrix}, & \dots & & \tilde{A}_{M^R-1}(x) & = \begin{bNiceMatrix}
		    * x^{N^R - 1}  \\ \Vdots \\ * x^{N^R - 1} \\ \alpha x^{N^R - 1} \\ * x^{N^R - 2} \\ \Vdots \\ * x^{N^R - 2}
		\end{bNiceMatrix},
	\end{aligned}
	\]
	where $\alpha$ is a nonzero constant and $*$ denotes a constant. For $\tilde{A}_{M^R-1}(x)$, the entry immediately following $\alpha$ is located at position $p-r+1$. These vector polynomials satisfy the relation $\tilde{A}(x)\Omega = R(x) A(x)$ for $n < M^R$.
    \end{itemize}
\end{Theorem}

\begin{proof}
    Define the polynomial family $\tilde{\tilde{B}}(x) \coloneqq \Omega B(x)$. By construction, the entries of $\Omega$ are chosen so that 
    \begin{equation*}
        \Omega B(\lambda_i) \vec{\boldsymbol{l}_i} = 0.
    \end{equation*}
    According to Theorem \ref{Theorem}, we conclude that $\tilde{\tilde{B}}(x)$ is divisible by $L(x)$, as both polynomials share the same spectral data. Thus, we can write
    \[
        \tilde{\tilde{B}}(x) = \tilde{B}(x) L(x),
    \]
    where $\tilde{B}(x)$ is an arbitrary matrix polynomial. Hence,
    \begin{equation*}
        \Omega B(x) = \tilde{B}(x) L(x).
    \end{equation*}
    Considering the first $q \times q$ block in this relation gives 
    \[
        \left[ \Omega_{0,0} \right]_q \mathcal{B}_0 + \dots + \left[ \Omega_{0,N^L-1} \right]_q \mathcal{B}_{N^L-1}(x) + \left[ \Omega_{0,N^L} \right]_q \mathcal{B}_{N^L}(x) = \tilde{\mathcal{B}}_0(x) L(x),
    \]
    where $\left[ \Omega_{i,j} \right]_q$ denotes a block of size $q \times q$ located at the $(i,j)$ entry of the matrix $\Omega$. As established in Theorem 5.2 of \cite{Manas_Rojas_Christoffel}, from this relation one deduces that $\tilde{\mathcal{B}}_0(x)$ is a lower triangular, invertible matrix independent of $x$. 

    Although not explicitly mentioned in \cite{Manas_Rojas_Christoffel}, the same argument can be generalized to analyze the degree structure of $\tilde{\mathcal{B}}_n(x)$ for $n \in \N_0$. Consequently, $B(x)$ and $\tilde{B}(x)$ share the same degree structure. By this, we mean that both $\mathcal{B}_n(x)$ and $\tilde{\mathcal{B}}_n(x)$ have diagonal entries of degree $n$, entries above the main diagonal of degree at most $n-1$, and entries below the diagonal of degree at most $n$.

    Next, consider the equality 
    \[
        \tilde{A}(x)\Omega = R(x) A(x),
    \]
    where no further assumptions are made on the family $\tilde{A}(x)$ apart from the third criterion of the theorem. In Theorem 3.1 of \cite{Manas_Rojas_Geronimus}, it was shown that the vectors from $\tilde{A}_{M^R}(x)$ to $\tilde{A}_{M^R+p-1}(x)$ are uniquely determined by the third criterion together with the equation $\tilde{A}(x)\Omega = R(x) A(x)$. Moreover, the corresponding vectors $\tilde{A}_{M^R}(x), \dots, \tilde{A}_{M^R+p-1}(x)$ share the same degree structure with the corresponding non-tilda vectors. 

    A minor oversight occurred in \cite{Manas_Rojas_Geronimus}: the conclusion that $\tilde{A}(x)$ and $A(x)$ share the same degree structure requires an intermediate argument. The argument is inductive: once the first $M^R$ vectors are known, one can construct the next $p$. Assuming now that the first $M^R + np$ vectors possess the correct degree structure, the same reasoning as in Theorem 3.1 of \cite{Manas_Rojas_Geronimus} allows construction of the next $p$ vectors with the required structure. This inductive step ensures that $\tilde{A}(x)$ and $A(x)$ indeed share the same degree structure.

    Finally, we establish the bi-orthogonality of these polynomial families:
    \[
        \int_\Delta \tilde{B}(x) \d \tilde{\mu}(x) \tilde{A}(x) \Omega = \Omega \int_\Delta B(x)L^{-1}(x) \d \tilde{\mu}(x) R(x) A(x) = \Omega \int_\Delta B(x) \d \mu(x) A(x) = \Omega.
    \]
    Hence,
    \[
        \int_\Delta \tilde{B}(x) \d \tilde{\mu}(x) \tilde{A}(x) \Omega = \Omega.
    \]
    If $\Omega$ admits an $LU$ factorization, we obtain
    \[
        \int_\Delta \tilde{B}(x) \d \tilde{\mu}(x) \tilde{A}(x) = I.   
    \]
    Therefore, $\tilde{B}(x)$ and $\tilde{A}(x)$ constitute bi-orthogonal families of polynomials with increasing degree structure, satisfying the orthogonality conditions.
\end{proof}

\begin{Remark}
    The argument used in Theorem 3.1 of \cite{Manas_Rojas_Geronimus} proceeds as follows: from the relation
    \[
        \tilde{\mathcal{A}}_0(x)\left[ \Omega_{0,0} \right]_p  + \dots + \tilde{\mathcal{A}}_{N^R-1}(x)\left[ \Omega_{N^R-1,0} \right]_p + \tilde{\mathcal{A}}_{N^R}(x)\left[ \Omega_{N^R,0} \right]_p = R(x)\mathcal{A}_0, 
    \]
    where $\left[ \Omega_{i,j} \right]_p$ denotes a block of size $p \times p$ located at the $(i,j)$ entry of $\Omega$, one proves that the first $p$ unknown entries of $\tilde{A}(x)$ are determined. The subsequent block in the relation $\tilde{A}(x)\Omega = R(x) A(x)$ yields
    \[
        \tilde{\mathcal{A}}_0(x)\left[ \Omega_{0,1} \right]_p  + \dots + \tilde{\mathcal{A}}_{N^R}(x)\left[ \Omega_{N^R,1} \right]_p + \tilde{\mathcal{A}}_{N^R+1}(x)\left[ \Omega_{N^R+1,1} \right]_p = R(x)\mathcal{A}_1. 
    \] 
    From this relation, and knowing the first $M^R+p$ vectors, one can determine the next $p$ vectors of $A(x)$ using the same reasoning as in Theorem 3.1 of \cite{Manas_Rojas_Geronimus}. This procedure extends to any $n$, provided that the previous $M^R + (n-1)p$ vectors possess the appropriate structure.
\end{Remark}

\begin{Remark}
    Consider an arbitrary $n \times n$ matrix $\Omega$ satisfying  
    \[
     M \Omega = \Omega, 
    \]
    where $M$ is free. Then $M = I_n$ if and only if $\Omega$ is invertible. For semi-infinite matrices, if $\Omega$ admits an $LU$ factorization we likewise obtain $M = I$. However, are there less restrictive conditions on $\Omega$ ensuring that $M = I$ is the only possibility? Moreover, is the assumption that $\Omega$ admits an $LU$ factorization independent, or could it be derived from the first and second criteria of the theorem? The authors do not know the answers to these questions, and for this reason we regard the result as partial.
\end{Remark}

The preceding theorem, combined with the Christoffel formulas previously derived, indicates that the existence of orthogonality is closely tied to the existence of explicit formulas for $\tilde{A}(x)$ and $\tilde{B}(x)$ for $n < M^R$. As noted in Remark \ref{Rm: n<M^R}, we have:
\begin{align*}
    \tilde{B}^{(b)}_0(x) & = \frac{1}{\tau_0} \sum_{\tilde{b}=1}^q\begin{vNiceMatrix}
        \mathbb{B}_0^{(1)} & \Cdots & \mathbb{B}_0^{(M^L)} & B_0^{(\tilde{b})}(x) \\
        \Vdots & & \Vdots & \Vdots \\
        \mathbb{B}_{M^L}^{(1)} & \Cdots & \mathbb{B}_{M^L}^{(M^L)} & B_{M^L}^{(\tilde{b})}(x)
    \end{vNiceMatrix} L^{-1}_{\tilde{b},b}(x), & \tau_0 & = \begin{vNiceMatrix}
        \mathbb{B}_0^{(1)} & \Cdots & \mathbb{B}_0^{(M^L)} \\
        \Vdots & & \Vdots \\
        \mathbb{B}_{M^L-1}^{(1)} & \Cdots & \mathbb{B}_{M^L-1}^{(M^L)}
    \end{vNiceMatrix}.
\end{align*}

For the case $n = 1$, we encounter $M^L + 1$ unknown components of $\Omega$, while $M^L$ equations are obtained from:
\[
    \begin{bNiceMatrix}
        \Omega_{1,0} & \Cdots & \Omega_{1,M^L}
    \end{bNiceMatrix} \begin{bNiceMatrix}
        \mathbb{B}_{0}^{(i)} \\ \Vdots \\ \mathbb{B}_{M^L}^{(i)}
    \end{bNiceMatrix} = -\mathbb{B}_{M^L+1}^{(i)}.
\]
An additional equation is derived by considering the relation:
\[
    \int_\Delta \tilde{B}_1(x) \, \mathrm{d} \tilde{\mu}(x) \, (X_{[p]}(x) \boldsymbol{e}_0) = 0,
\]
where $\boldsymbol{e}_0 = \left[ \begin{smallmatrix} 1 & 0 & 0 & \cdots \end{smallmatrix} \right]$. This leads to:
\[
    0 = \int_\Delta \tilde{B}_1(x)  L(x) \left( \mathrm{d} \mu(x) R^{-1}(x) + \sum_{i=1}^{M^R}\boldsymbol{\xi}_{i}(x) \cev{\boldsymbol{r}}_{i}\delta(x-\rho_i)  \right)(X_{[p]}(x) \boldsymbol{e}_0).
\]
\begin{Definition}
    We introduce the following constants:
    \begin{equation*}
        \mathbb{I}_{i,j} = \int_\Delta B_i(x) \left( \d \mu(x) R^{-1}(x) + \sum_{i=1}^{M^R}\boldsymbol{\xi}_{i}(x) \cev{\boldsymbol{r}}_{i}\delta(x-\rho_i)  \right)(X_{[p]}(x) \boldsymbol{e}_j).
    \end{equation*}
\end{Definition}

With this definition, the previous relation can be rewritten as:
\begin{multline*}
    0 = \begin{bNiceMatrix}
        \Omega_{1,0} & \Cdots & \Omega_{1,M^L}
    \end{bNiceMatrix} \int_\Delta \begin{bNiceMatrix}
        B_0(x) \\ \Vdots \\ B_{M^L}(x)
    \end{bNiceMatrix} \left( \d \mu(x) R^{-1}(x) + \sum_{i=1}^{M^R} \boldsymbol{\xi}_i(x) \cev{\boldsymbol{r}}_i \delta(x - \rho_i) \right)(X_{[p]}(x) \boldsymbol{e}_0) \\
    + \int_\Delta B_{M^L+1}(x) \left( \d \mu(x) R^{-1}(x) + \sum_{i=1}^{M^R} \boldsymbol{\xi}_i(x) \cev{\boldsymbol{r}}_i \delta(x - \rho_i) \right)(X_{[p]}(x) \boldsymbol{e}_0),
\end{multline*}
which leads to the linear equation:
\begin{equation*}
    \begin{bNiceMatrix}
        \Omega_{1,0} & \Cdots & \Omega_{1,M^L}
    \end{bNiceMatrix} \begin{bNiceMatrix}
        \mathbb{I}_{0,0} \\ \Vdots \\ \mathbb{I}_{M^L,0}
    \end{bNiceMatrix} = - \mathbb{I}_{M^L+1,0}.
\end{equation*}

This provides a linear system with as many equations as unknowns, and the components of the matrix $\Omega$ are determined via:
\begin{equation*}
    \begin{bNiceMatrix}
        \Omega_{1,0} & \Cdots & \Omega_{1,M^L}
    \end{bNiceMatrix} \begin{bNiceMatrix}
        \mathbb{B}_{0}^{(1)} & \Cdots & \mathbb{B}_{0}^{(M^L)} & \mathbb{I}_{0,0} \\ \Vdots & & \Vdots & \Vdots \\ 
        \mathbb{B}_{M^L}^{(1)} & \Cdots & \mathbb{B}_{M^L}^{(M^L)} & \mathbb{I}_{M^L,0}
    \end{bNiceMatrix} = - \begin{bNiceMatrix}
        \mathbb{B}_{M^L+1}^{(1)} & \Cdots & \mathbb{B}_{M^L+1}^{(M^L)} & \mathbb{I}_{M^L+1,0}
    \end{bNiceMatrix}.
\end{equation*}

This yields a result analogous to the case $n \geq M^R$: the system is solvable provided that the determinant
\begin{equation*}
    \tau_1 \coloneqq \begin{vNiceMatrix}
         \mathbb{B}_{0}^{(1)} & \Cdots & \mathbb{B}_{0}^{(M^L)} & \mathbb{I}_{0,0} \\ \Vdots & & \Vdots & \Vdots \\ 
        \mathbb{B}_{M^L}^{(1)} & \Cdots & \mathbb{B}_{M^L}^{(M^L)} & \mathbb{I}_{M^L,0}
    \end{vNiceMatrix} \neq 0.
\end{equation*}

Note that the $\tau$-determinants were previously undefined for the regime $n < M^R$. For consistency, we maintain the same notation.

For general $n < M^R$, the orthogonality condition
\begin{align*}
    \int_\Delta \tilde{B}_n(x) \, \d \tilde{\mu}(x) \, (X_{[p]}(x) \boldsymbol{e}_j) &= 0, & \text{for} \quad j &\in \{ 0, \dots , n-1 \},
\end{align*}
leads to the linear system:
\begin{multline*}
    \begin{bNiceMatrix}
        \Omega_{n,0} & \Cdots & \Omega_{1,M^L+n}
    \end{bNiceMatrix} \begin{bNiceMatrix}
        \mathbb{B}_{0}^{(1)} & \Cdots & \mathbb{B}_{0}^{(M^L)} & \mathbb{I}_{0,0} & \Cdots & \mathbb{I}_{0,n-1} \\ \Vdots & & \Vdots & \Vdots & & \Vdots  \\ 
        \mathbb{B}_{M^L+n-1}^{(1)} & \Cdots & \mathbb{B}_{M^L+n-1}^{(M^L)} & \mathbb{I}_{M^L+n-1,0} & \Cdots & \mathbb{I}_{M^L+n-1,n-1}
    \end{bNiceMatrix} \\
    = - \begin{bNiceMatrix}
        \mathbb{B}_{M^L+n}^{(1)} & \Cdots & \mathbb{B}_{M^L+n}^{(M^L)} & \mathbb{I}_{M^L+n,0} & \Cdots & \mathbb{I}_{M^L+n,n-1}
    \end{bNiceMatrix},
\end{multline*}
with corresponding determinant
\begin{equation} \label{Eq: Tau_M^R}
    \tau_n \coloneqq \begin{vNiceMatrix}
        \mathbb{B}_{0}^{(1)} & \Cdots & \mathbb{B}_{0}^{(M^L)} & \mathbb{I}_{0,0} & \Cdots & \mathbb{I}_{0,n-1} \\ \Vdots & & \Vdots & \Vdots & & \Vdots  \\ 
        \mathbb{B}_{M^L+n-1}^{(1)} & \Cdots & \mathbb{B}_{M^L+n-1}^{(M^L)} & \mathbb{I}_{M^L+n-1,0} & \Cdots & \mathbb{I}_{M^L+n-1,n-1}
    \end{vNiceMatrix}.
\end{equation}

\begin{Proposition}
    Theorem \ref{Th: Existence} extends to the $\tau$-determinants defined in Equation \eqref{Eq: Tau_M^R}. Specifically, if orthogonality exists, then $\tau_n \neq 0$ for all $n \in \mathbb{N}_0$.
\end{Proposition}

\begin{proof}
    Consider the case $n < M^R - 1$. If orthogonality holds, then
    \begin{equation*}
        \int_\Delta \tilde{B}_n(x) \, \d \tilde{\mu}(x) \, (X_{[p]}(x) \boldsymbol{e}_n) = \alpha_n,
    \end{equation*}
    for some $\alpha_n \neq 0$. Substituting the expressions for the components of $\Omega$, we obtain:
    \begin{multline*}
        \mathbb{I}_{n,n} - \begin{bNiceMatrix}
            \mathbb{B}_{M^L+n}^{(1)} & \Cdots & \mathbb{B}_{M^L+n}^{(M^L)} & \mathbb{I}_{M^L+n,0} & \Cdots & \mathbb{I}_{M^L+n,n-1} 
        \end{bNiceMatrix} \\
        \times 
        \begin{bNiceMatrix}
        \mathbb{B}_{0}^{(1)} & \Cdots & \mathbb{B}_{0}^{(M^L)} & \mathbb{I}_{0,0} & \Cdots & \mathbb{I}_{0,n-1} \\ \Vdots & & \Vdots & \Vdots & & \Vdots  \\ 
        \mathbb{B}_{M^L+n-1}^{(1)} & \Cdots & \mathbb{B}_{M^L+n-1}^{(M^L)} & \mathbb{I}_{M^L+n-1,0} & \Cdots & \mathbb{I}_{M^L+n-1,n-1}
    \end{bNiceMatrix}^{-1} \begin{bNiceMatrix}
        \mathbb{I}_{0,n} \\
        \Vdots \\
        \mathbb{I}_{n-1,n} 
    \end{bNiceMatrix} = \frac{\tau_{n+1}}{\tau_{n}}.
    \end{multline*}
    Since $\alpha_n \neq 0$, it follows that all $\tau_n \neq 0$. In the special case $n = M^R - 1$, we have $\alpha_{M^R-1} = \frac{\tau_{M^R}}{\tau_{M^R-1}}$, where both expressions for $\tau_{M^R}$, the one from Definition \ref{Def: Taudeter} and that from Equation \eqref{Eq: Tau_M^R}, solve the same linear system and thus must agree.
\end{proof}

The non-vanishing of these $\tau_n$ determinants for $n < M^R$ implies the existence of a sequence of vector polynomials $\check{B}_n(x)$ satisfying:
\[
    \begin{aligned}
	\int_\Delta \sum_{b=1}^q \tilde{B}_n^{(b)}(x) \, \mathrm{d}\tilde{\mu}_{b,a}(x) \, x^l &= 0, & a &\in \{1,\dots,p\}, & l &\in \left\{0,\dots, \left\lceil\frac{n-a+2}{p}\right\rceil-1\right\},
    \end{aligned}
\]
for all $n \in \{0, \dots, M^R - 1\}$. These orthogonality relations ensure the existence of $M^R$ vector polynomials $\tilde{A}(x)$ satisfying the second condition in Theorem \ref{Th: Existence II}. We thus establish the following result:

\begin{Theorem}
    Perturbed orthogonality exists if the following conditions are satisfied:
    \begin{itemize}
        \item $\tau_n \neq 0$ for all $n \in \mathbb{N}_0$;
        \item The connection banded matrix $\Omega$  admits an $LU$ factorization.
    \end{itemize}
\end{Theorem}

\section*{Summary and Outlook}

This work has developed a matrix-analytic framework for general Uvarov-type perturbations of mixed-type multiple orthogonal polynomials on the step line.
By introducing regular matrix polynomials acting on both sides of a rectangular matrix of measures and allowing for rational and additive modifications, the Christoffel, Geronimus, and Uvarov transformations have been unified within a single algebraic and spectral scheme.
Explicit connection formulas, determinant representations, and $\tau$–determinant conditions ensuring the existence of perturbed orthogonality have been obtained.

The analysis has revealed the intrinsic matrix structure of these transformations, their interpretation as structured algebraic and spectral modifications of block-banded operators, and the central role of the Markov–Stieltjes matrix function in encoding such deformations.
The equivalence between the non-vanishing of the $\tau$–determinants and the actual existence of the Uvarov-type transformation, however, remains an open question.
While in the Christoffel and Geronimus cases sufficiency follows respectively from divisibility arguments for matrix polynomials and from spectral considerations, in the Uvarov case sufficiency has been established only under the additional hypothesis that the connection matrix $\Omega$ admits a Gauss–Borel (or $LU$) factorization.
Removing this structural assumption—or alternatively constructing an explicit counterexample—constitutes a central open problem for future research.

Possible extensions of the present work include the development of a general theory of Uvarov-type perturbations beyond the step-line configuration, and the formulation of analogous constructions for multivariate and matrix-variate orthogonal polynomials.
These directions are naturally connected with earlier studies on multivariate orthogonal polynomials and their Christoffel transformations~\cite{GM1,GM2}, as well as with ongoing research on bivariate multiple orthogonal polynomials of mixed type on the step line~\cite{GMJ}.
Both avenues point toward a broader algebraic framework encompassing higher-dimensional settings, where the interplay between matrix analysis, orthogonality, and integrable systems becomes even richer.
These perspectives will be pursued in subsequent research currently in progress.

	\section*{Acknowledgments}

The authors acknowledges research project [PID2021- 122154NB-I00], \emph{Ortogonalidad y Aproximación con Aplicaciones en Machine Learning y Teoría de la Probabilidad}  funded  by
\href{https://doi.org/10.13039/501100011033}{MICIU/AEI/10.13039 /501100011033} and by "ERDF A Way of making Europe” and   PID2024-155133NB-I00,  \emph{Ortogonalidad, Aproximación e Integrabilidad: Aplicaciones en Procesos Estocásticos Clásicos y Cuánticos}.

\section*{Declarations}

\begin{enumerate}
	\item \textbf{Conflict of interest:} The authors declare no conflict of interest.
	\item \textbf{Ethical approval:} Not applicable.
	\item \textbf{Contributions:} All the authors have contribute equally.
\end{enumerate}

\underline{Corresponding author:} Manuel Mañas, manuel.manas@ucm.es

\end{document}